\documentclass[reqno]{amsart}
\usepackage{etex}
\usepackage[a4paper,hmarginratio=1:1]{geometry}
\usepackage{amssymb,amsfonts,amsmath}
\usepackage{comment} 
\usepackage[all,arc]{xy}
\usepackage{enumerate}
\usepackage{mathrsfs,mathtools}
\usepackage{todonotes,booktabs}
\usepackage{stmaryrd}
\usepackage{marvosym}
\usepackage{graphicx}
\usepackage{pdflscape} 	
\usepackage{bbm}

\usepackage{tikz}
\usepackage{tikz-cd}
\usetikzlibrary{trees}
\usetikzlibrary[shapes]
\usetikzlibrary[arrows]
\usetikzlibrary{patterns}
\usetikzlibrary{fadings}
\usetikzlibrary{backgrounds}
\usetikzlibrary{decorations.pathreplacing}
\usetikzlibrary{decorations.pathmorphing}
\usetikzlibrary{positioning}
	
\def\biblio{\bibliography{duality}\bibliographystyle{alpha}}

\usepackage{xcolor} 
\usepackage{graphicx}
\SelectTips{cm}{10}

\usepackage[pagebackref]{hyperref}

\definecolor{dark-red}{rgb}{0.5,0.15,0.15}
\definecolor{dark-blue}{rgb}{0.15,0.15,0.6}
\definecolor{dark-green}{rgb}{0.15,0.6,0.15}
\hypersetup{
    colorlinks, linkcolor=dark-red,
    citecolor=dark-blue, urlcolor=dark-green
}

\newcommand{\cX}{\cal{X}}

\newcommand{\cW}{\cal{W}}

\renewcommand*{\backref}[1]{}
\renewcommand*{\backrefalt}[4]{%
  \ifcase #1 %
No citations.
  \or
(cit. on p. #2).%
  \else
(cit on pp. #2).%
  \fi%
}

\usepackage[nameinlink,capitalise,noabbrev]{cleveref}

\newtheorem{thm}{Theorem}[section]
\newtheorem{cor}[thm]{Corollary}
\newtheorem{prop}[thm]{Proposition}
\newtheorem{lem}[thm]{Lemma}

\theoremstyle{definition}
\newtheorem{defn}[thm]{Definition}

\newtheorem{ex}[thm]{Example}

\theoremstyle{remark}
\newtheorem{rem}[thm]{Remark}
 
\theoremstyle{theorem}
\newtheorem*{thm*}{Theorem}
\newtheorem*{thma}{Theorem A}
\newtheorem*{thmb}{Theorem B}
\newtheorem*{thmc}{Theorem C}
\newtheorem*{thmd}{Theorem D}
\newtheorem*{thme}{Theorem E}
\newtheorem*{quest*}{Question}

\makeatletter
\let\c@equation\c@thm
\makeatother
\numberwithin{equation}{section}

\DeclareMathOperator{\Top}{Top}

\DeclareMathOperator{\Sp}{Sp}

\DeclareMathOperator{\Hom}{Hom}

\DeclareMathOperator{\End}{End}

\DeclareMathOperator{\cA}{\mathcal{A}}
\DeclareMathOperator{\cC}{\mathcal{C}}
\DeclareMathOperator{\cD}{\mathcal{D}}
\DeclareMathOperator{\cE}{\mathcal{E}}

\DeclareMathOperator{\cS}{\mathcal{S}}

\DeclareMathOperator{\cF}{\mathcal{F}}

\DeclareMathOperator{\cV}{\mathcal{V}}

\DeclareMathOperator{\cO}{\mathcal{O}}

\DeclareMathOperator{\Spec}{Spec}
\DeclareMathOperator{\Mod}{Mod}

\DeclareMathOperator{\StMod}{StMod}

\DeclareMathOperator{\Loc}{Loc}
\DeclareMathOperator{\Coloc}{Coloc}

\DeclareMathOperator{\Thick}{Thick}
\DeclareMathOperator{\Ind}{Ind}

\DeclareMathOperator{\res}{res}

\DeclareMathOperator{\im}{Im}

\DeclareMathOperator{\Res}{Res}
\DeclareMathOperator{\Coind}{Coind}

\DeclareMathOperator{\supp}{supp}
\DeclareMathOperator{\cosupp}{cosupp}

\newcommand{\cK}{\mathcal{K}}

\newcommand{\cal}{\mathcal}
\newcommand{\xr}{\xrightarrow}

\newcommand{\Z}{\mathbb{Z}}

\Crefname{figure}{Figure}{Figures}
\Crefname{assu}{Assumption}{Assumptions}
\Crefname{lem}{Lemma}{Lemmas}
\Crefname{thm}{Theorem}{Theorems}
\Crefname{thma}{Theorem}{Theorems}
\Crefname{prop}{Proposition}{Propositions}

\DeclareMathOperator{\Inj}{Inj}

\newcommand{\fp}{\mathfrak{p}}
\newcommand{\fq}{\mathfrak{q}}
\newcommand{\fr}{\mathfrak{r}}
\newcommand{\fS}{\mathfrak{S}}

\newcommand{\recollement}[5]{
\xymatrix{{#1} \ar[r]|-{#2} & #3 \ar[r]|-{#4} \ar@<1ex>[l]^-{{#2}_!} \ar@<-1ex>[l]_-{{#2}^*} & #5, \ar@<1ex>[l]^-{{#4}!} \ar@<-1ex>[l]_-{{#4}^*}
}}
\let\lim\relax

\DeclareMathOperator{\lim}{lim}

\newcommand{\mm}{/\!\!/}

\DeclareMathOperator{\op}{op}

\newcommand{\cQ}{\mathcal{Q}}
\newcommand{\cR}{\mathcal{R}}
\newcommand{\cZ}{\mathcal{Z}}
\newcommand{\cU}{\mathcal{U}}

\newcommand{\F}{\mathbb{F}}

\newcommand{\pc}[1]{{#1}_{p}^{\wedge}}

\DeclareMathOperator{\map}{map}

\title{On stratification for spaces with Noetherian mod $p$ cohomology}

\author{Tobias Barthel}
\address{Max-Planck-Institut f\"ur Mathematik, Vivatsgasse 7, 53111 Bonn, Germany}
\email{tbarthel@mpim-bonn.mpg.de}

\author{Nat{\`a}lia Castellana}
\address{Departament de Matem\`atiques, Universitat Aut\`onoma de Barcelona, 08193 Cerdanyola del Vall\'es (Barcelona), Spain and Centre de Recerca Matem\`atica, Edifici Cc, Campus de Bellaterra, 08193 Cerdanyola del Vall\'es (Barcelona), Spain}
\email{natalia@mat.uab.cat}

\author{Drew Heard}
\address{Department of Mathematical Sciences, Norwegian University of Science and Technology, Trondheim}
\email{drew.k.heard@ntnu.no}

\author{Gabriel Valenzuela}
\address{Max-Planck-Institut f\"ur Mathematik, Vivatsgasse 7,
53111 Bonn,
Germany}
\email{gvalezuela@mpim-bonn.mpg.de}

\date{\today}

\begin{document}

\begin{abstract}
Let $X$ be a topological space with Noetherian mod $p$ cohomology and let $C^*(X;\mathbb{F}_p)$ be the commutative ring spectrum of $\mathbb{F}_p$-valued cochains on $X$. The goal of this paper is to exhibit conditions under which the category of module spectra on $C^*(X;\mathbb{F}_p)$ is stratified in the sense of Benson, Iyengar, Krause, providing a classification of all its localizing subcategories. We establish stratification in this sense for classifying spaces of a large class of topological groups including Kac--Moody groups as well as whenever $X$ admits an $H$-space structure. More generally, using Lannes' theory we prove that stratification for $X$ is equivalent to a condition that generalizes Chouinard's theorem for finite groups. In particular, this relates the generalized telescope conjecture in this setting to a question in unstable homotopy theory. 
\vspace{-2em}
\end{abstract}

\maketitle

\setcounter{tocdepth}{1}
\tableofcontents
\vspace{-2em}
\def\biblio{}

\section*{Introduction}\label{sec:introduction}

\subsection*{Background and motivation}

In a series of papers \cite{benson_local_cohom_2008,bik11} culminating in \cite{bik_finitegroups}, Benson, Iyengar, and Krause classified all modular representations of a finite group $G$ up to a certain equivalence relation: Two representations are considered equivalent if they can be constructed from each other via the tensor-triangulated structure available on the stable module category $\StMod_{\F_pG}$ of $G$. More precisely, they show that the localizing subcategories of $\StMod_{\F_pG}$ are parametrized by subsets of the projective variety of the group cohomology ring $H^*(G;\F_p)$. 

Earlier, a similar classification theorem was obtained by Neeman \cite{neemanchromtower} for the derived category of a Noetherian commutative ring $A$; in this case, localizing subcategories are parametrized by the Zariski spectrum of $A$. Both classification problems admit a common formulation in the setting of stable homotopy theory, as noticed first by Hovey, Palmieri, and Strickland \cite{hps_axiomatic}: Given a commutative ring spectrum $R$, we may ask when the localizing subcategories of the category $\Mod_R$ of $R$-module spectra are parametrized by the Zariski spectrum $\Spec(\pi_*R)$ via certain support functions. When this is the case, we follow Benson, Iyengar, and Krause and say that $\Mod_R$ is stratified or simply that stratification holds for $R$. It turns out that, in general, stratification is only reasonable to expect when $R$ is Noetherian, i.e., when $\pi_*R$ is a graded Noetherian ring, so we will work under this assumption from now on.

Neeman's theorem is then the special case that $R = HA$, the Eilenberg--Mac Lane ring spectrum of a Noetherian commutative ring $A$. Another important source of Noetherian commutative ring spectra comes from unstable homotopy theory, given by commutative ring spectra of $\F_p$-valued cochains $C^*(X;\F_p)$ on a topological space $X$ with Noetherian mod $p$ cohomology. In this language, the theorem of Benson, Iyengar, and Krause establishes stratification for cochains on $X= BG$, the classifying space of a finite group $G$. Their work has subsequently been extended by Benson and Greenlees \cite{bg_stratifyingcompactlie} and more recently in \cite{bchv1} for classifying spaces of certain topological groups as well as by Shamir \cite{shamir_strat} for certain spaces constructed from iterated spherical fibrations. This motivates the following guiding question:

\begin{quest*}
For which spaces $X$ with Noetherian mod $p$ cohomology ring is the category $\Mod_{C^*(X;\F_p)}$ of module spectra over the commutative ring spectrum of cochains $C^*(X;\F_p)$ stratified?
\end{quest*}

The goal of the present paper is to establish stratification in this sense for cochains on a large class of topological groups as well as for $H$-spaces with Noetherian mod $p$ cohomology, and to provide a provide a systematic approach to answering this question in general.

\subsection*{Main results}

We first consider commutative ring spectra of cochains on classifying spaces of  topological groups, extending our previous work in \cite{bchv1}. By reducing to the case of compact Lie groups via work of Broto and Kitchloo \cite{broto_kitchloo}, we in particular establish stratification for Kac--Moody groups. 

\begin{thma}[\Cref{thm:kacmoodystrat}]\label{thma}
If $K$ is a Kac--Moody group, then $\Mod_{C^*(BK;\F_p)}$ is stratified. 
\end{thma}

Hopf spaces with Noetherian mod $p$ cohomology form the second class of examples. Our proof of stratification in this case relies crucially on several deep results in unstable homotopy theory stemming from the resolution of the Sullivan conjecture, as outlined below. In particular, we make use of a structure theorem for $H$-spaces due to Broto, Crespo, and Saumell \cite{bcs_deconstructing} and Castellan, Crespo, and Scherer \cite{ccs_deconstructing}. 

\begin{thmb}[\Cref{thm:hnoetherian}]\label{thmb}
If $X$ is a connected $H$-space with Noetherian mod $p$ cohomology, then $\Mod_{C^{*}(X;\F_p)}$ is stratified.
\end{thmb}

In particular, as a result of the general theory of stratifications, the generalized telescope conjecture holds for the module categories of Theorems A and B. 

Finally, we formulate an approach to answering the above question for any space with Noetherian mod $p$ cohomology, based again on Lannes' theory \cite{La} and its consequences. This allows us to reduce this categorical classification problem to a condition that suitable generalizes Chouinard's theorem \cite{chouinard} about projective representations for finite group; the terminology will be introduced more carefully in the next subsection. 

\begin{thmc}[\Cref{thm:rectorstrat}]
Let $X$ be a $p$-good connected space with Noetherian mod $p$ cohomology, then $\Mod_{C^*(X;\F_p)}$ is stratified if and only if $X$ satisfies Chouinard's condition.
\end{thmc}

In particular, we verify Chouinard's condition for $H$-spaces with Noetherian mod $p$ cohomology, thereby recovering Theorem B as an instance of Theorem C.

\subsection*{Outline of the methods}

From our perspective, stable homotopy theory provides a convenient context for developing an abstract descent theory for stratifications that places the results mentioned above in a uniform setting. Consider thus a morphism $f\colon R \to S$ of Noetherian commutative ring spectra, i.e., homotopically coherent ring spectra with Noetherian graded commutative coefficient rings. Inspired by \cite{bik_finitegroups} and \cite{bg_stratifyingcompactlie}, in \cite{bchv1} we isolated two key conditions on $f$ that guarantee that stratification descends from $\Mod_S$ to $\Mod_R$ along $f$: 
	\begin{enumerate}
		\item Induction and coinduction along $f$ give conservative functors $\Mod_R \to \Mod_S$, i.e., they reflect isomorphisms. In this case, we say that $f$ is biconservative.
		\item $f$ satisfies a support theoretic condition  called Quillen lifting. 
	\end{enumerate}
Geometrically, these conditions correspond, roughly speaking, to the surjectivity and injectivity of the induced map on homogeneous Zariski spectra $\res_f\colon \Spec^h(\pi_*S) \to \Spec^h(\pi_*R)$.	
	
In \cite{bchv1}, we showed that Conditions (1) and (2) are sufficient to descent stratification from $\Mod_S$ to $\Mod_R$. The key novelty in this result was that $f$ was not assumed to satisfy any additional finiteness conditions, in contrast to the situation studied in the aforementioned works \cite{bik_finitegroups, bg_stratifyingcompactlie}. This opened up the possibility to apply such descent techniques to many new examples. 

Our first aim in this paper is to characterize those biconservative morphisms $f\colon R \to S$ that allow descent of stratification. In particular, we complement \cite{bchv1} by showing that Condition (2) is also necessary for descent; in fact, we exhibit several equivalent characterizations of this condition. Informally speaking, this shows that our descent theorem is optimal: 

\begin{thmd}[\Cref{thm:stratification} and \Cref{cor:as}]\label{thmd}
Let $f\colon R \to S$ be a biconservative morphism of Noetherian commutative ring spectra and assume $\Mod_S$ is stratified, then the following are equivalent:
	\begin{enumerate}
		\item $\Mod_R$ is stratified.
		\item The morphism $f$ satisfies Quillen lifting. 
		\item Every module $M \in \Mod_R$ satisfies the following Avrunin--Scott identities:
\[
\supp_{S}(f^* M) = \res_f^{-1}\supp_{R}(M) \quad \text{and} \quad
\cosupp_S(f^! M) = \res_f^{-1}\cosupp_R(M),
\]
		where $f^*$ and $f^!$ denote induction and coinduction along $f$, respectively. 
	\end{enumerate}
\end{thmd}
Given a commutative ring spectrum $R$, this theorem thus reduces the problem of stratifying $\Mod_R$ to the construction of a suitable morphism $f\colon R \to S$. By analogy with Chouinard's theorem for discrete groups, for the remainder of this introduction we will say that a morphism $f\colon R \to S$ satisfies Chouinard's condition\footnote{Note that this terminology is reserved for a more restrictive situation in the main body of the paper.} whenever $f$ is biconservative.

To prove Theorem D, we introduce and study a stronger version of Quillen lifting called simple Quillen lifting that in practice is easier to check for a given example. Furthermore, we relate both Chouinard's condition as well as simple Quillen lifting to properties of the induced morphism 
\[
\xymatrix{\res_f\colon \Spec^h(\pi_*S) \ar[r] & \Spec^h(\pi_*R).}
\]

Our main theorems are then an application of this general descent theorem; however, in each case additional input from unstable homotopy theory is required. Theorem A follows the line of arguments in our previous paper \cite{bchv1} and earlier work of Benson and Greenlees \cite{bg_stratifyingcompactlie} on stratification for topological groups $G$. The idea is to descent stratification along a morphism of commutative ring spectra
\[
\xymatrix{\rho_{G}\colon C^*(BG;\F_p) \ar[r] & \prod_{E \in \cal{E}_p(G)}C^*(BE;\F_p),}
\]
where the product is indexed on a set of representatives of conjugacy classes of elementary abelian $p$-subgroups of $G$. By virtue of Theorem D, it suffices to show that $\rho_{G}$ satisfies Quillen lifting and Chouinard's condition. In the case of a Kac--Moody group, the first property follows from a generalization of Quillen's $F$-isomorphism to Kac--Moody groups due to Broto and Kitchloo \cite{broto_kitchloo}, see \Cref{thm:kmhfiso}. The verification of Chouinard's condition for $\rho_{G}$ uses the construction of an appropriate compactly generated $G$-space $X$ equipped with specified isotropy, which may be thought of as a suitable replacement of the morphism $BG \to BU(n)$ induced by a unitary embedding of a compact Lie group $G$. In fact, our methods apply more generally to spaces occurring in the Broto--Kitchloo hierarchy of compact Lie groups \cite{kropholler_mislin}, see \Cref{thm:stratk1x}. 

We would like to mimic this approach for an arbitrary connected space $X$ with Noetherian mod $p$ cohomology. However, we are immediately confronted with the problem of finding a homotopical analogue of the morphism $\rho_{G}$ for $X$, which in turn relies on finding a replacement of the notion of elementary abelian subgroup for $X$. We resolve this problem by making use of Lannes' \cite{lannes1985, La} and Miller's \cite{miller_sullivan} work on the Sullivan conjecture to construct a morphism of commutative ring spectra
\[
\xymatrix{\rho_X\colon C^*(X;\F_p) \ar[r] & \prod_{(E,\varphi)\in \cE(X)} C^*(BE;\F_p).}
\]
The product in the target is indexed on homotopy classes of maps $\varphi\colon BE \to X$ which induce finite algebra morphisms in mod $p$ cohomology.\footnote{For simplicity, we omit $p$-completions throughout the introduction.} For $X = BG$ the classifying space of a finite $p$-group, $\rho_X$ specializes to the restriction functor featuring in the work of Chouinard, Quillen, as well as Benson--Iyengar--Krause. We thus view $\rho_X$ as a suitable catalyst for our descent techniques; in particular, we say that $X$ satisfies Chouinard's condition if $\rho_X$ does. 

Lannes' theory provides a bijection between the set $\cE(X)$ appearing in the target of $\rho_X$ and the objects in the Rector category of $H^*(X;\F_p)$, a category which generalizes Quillen's category for $X= BG$. An elaboration of this observation combined with Rector's work \cite{rector_quillenstrat} then proves that $\rho_X$ satisfies Quillen lifting for any $p$-good connected space $X$ with Noetherian mod $p$ cohomology, see \Cref{prop:autoql}. In light of Theorem D, this establishes Theorem C.

We then demonstrate this framework in the case of $H$-spaces $X$ with Noetherian mod $p$ cohomology. By Theorem C, in order to prove stratification for $\Mod_{C^*(X;\F_p)}$, it remains to show that $\rho_X$ satisfies Chouinard's condition. To this end, we make use of a theorem due to Broto, Crespo, and Saumell \cite{bcs_deconstructing} and Castellana, Crespo, and Scherer \cite{ccs_deconstructing}, according to which any $H$-space $X$ with Noetherian mod $p$ cohomology may be ``decomposed'', after $p$-completion, into the classifying space of an abelian compact Lie group and a nilpotent space with finite mod $p$-cohomology. Combined with work of Henn, Lannes, and Schwartz \cite{hls_categoryU}, we arrive at:

\begin{thme}[\Cref{thm:nhchouinard}]\label{thme}
Any connected $H$-space $X$ with Noetherian mod $p$ cohomology satisfies Chouinard's condition, i.e., induction and coinduction along $\rho_X$ are conservative.
\end{thme}

Theorem B is now an immediate consequence of Theorem C and Theorem E. In fact, we also give a proof of Theorem B that does not rely on Theorem E, but instead applies Theorem C directly to the fiber sequence provided by the deconstruction of $H$-spaces, see the proof of \Cref{thm:hnoetherian}. In particular, our result applies to $S^3\langle 3 \rangle$, the 3-connected cover of $S^3$, to provide an example of a commutative ring spectrum for which we deduce stratification, but are currently unable to prove costratification.

\subsection*{Organization of the paper}

The first two sections deal with the interaction between the stratification theory of Benson--Iyengar--Krause and morphisms of commutative ring spectra, culminating in the proof of our main descent theorem. Section 1 starts with a review of support theory and then studies base-change and conservativity in this context. The subject of Section 2 is the proof of Theorem D. Along the way, we establish the main properties of Quillen lifting and simple Quillen lifting, augmented by some examples. 

From Section 3 onwards, we specialize to cochains on spaces with Noetherian mod $p$ cohomology. Theorems A, B, and C are proven in Sections 3, 4, and 5, respectively. Each section begins with a review of the necessary background material and concludes with explicit examples. 

\subsection*{Notation and conventions}

\begin{itemize}
	\item Let $R$ be a commutative ring spectrum and write $\Mod_R$ for the category of module spectra over $R$, where the symmetric monoidal structure is given by the relative smash product $\otimes = \otimes_R$ with unit $R$. If $\cS$ is a set of objects in $\Mod_R$, then the smallest thick subcategory of $\Mod_R$ containing 
	$\cS$ will be denoted by $\Thick_R(\cS)$, and the smallest localizing subcategory of $\Mod_R$ containing $\cS$ will be denoted $\Loc_R(\cS)$. 
	If $R$ is clear from context, we sometimes omit the corresponding subscript. 
	\item We write $H^*(X)$ for the mod $p$ cohomology of a space $X$. 
  \item The $p$-completion of a space $X$ is always meant in the sense of Bousfield--Kan \cite{bousfield_kan}, and is denoted $\pc{X}$.  A space $X$ is called $p$-complete if the $p$-completion map $X \to \pc{X}$ is a homotopy equivalence, and is called $\F_p$-finite if $H^*(X)$ is finite. A space is $p$-good if the natural completion map $X \to \pc{X}$ induces an isomorphism on mod $p$ cohomology.
  \item We write $\res_{f}\colon \Spec^h(\pi_*B) \to \Spec^h(\pi_*A)$ for the restriction induced by a morphism of commutative ring spectra $f\colon A \to B$, where $\Spec^h$ denotes the homogeneous spectrum of prime ideals.  
  \item A functor $F\colon \cC \to \cD$ is said to be conservative if it detects equivalences; if $F$ is an exact functor between stable categories, then this is equivalent to $F$ detecting the zero object, i.e., $F(X) \simeq 0$ implies $X \simeq 0$ for any $X \in \cC$. 
\end{itemize}

\subsection*{Acknowledgements}

We would like to thank Jesper Grodal and Julia Pevtsova for useful conversations, and the referee for many helpful comments on an earlier version of this paper. The first author was partially supported by the DNRF92 and the European Unions Horizon 2020 research and innovation programme under the Marie Sklodowska-Curie grant agreement No.~751794, the second author was partially supported by FEDER-MEC grant MTM2016-80439-P, the third  author was partially supported by the SFB 1085 ``Higher Invariants", and the fourth author would like to thank the Max Planck Institute for Mathematics for its hospitality. The first and second author would furthermore like to thank the Isaac Newton Institute for Mathematical Sciences, Cambridge, for support and hospitality during the programme \emph{Homotopy Harnessing Higher Structures}, where work on this paper was undertaken. This work was supported by EPSRC grant no EP/K032208/1.

\section{Support and base-change}\label{sec:support}
In this section we review the support theory introduced by Benson, Iyengar, and Krause \cite{benson_local_cohom_2008}, specialized to the setting of structured ring spectra, and show how this support theory interacts with a morphism $f \colon R \to S$ of ring spectra. 
\subsection{Support theory and stratification}

Let $R$ be a Noetherian commutative ring spectrum and fix a homogeneous prime ideal $\mathfrak p \in \Spec^h(\pi_*R)$. Associated to $\mathfrak p$, Benson, Iyengar, and Krause \cite{benson_local_cohom_2008,bik12} have constructed local cohomology and local homology functors on $\Mod_R$, denoted $\Gamma_{\fp}$ and $\Lambda^{\frak p}$ respectively. In the context of structured ring spectra, these have been studied in detail in \cite{bhv2,bchv1}, from where we briefly recall their definition. For $\fr\in \Spec^h(\pi_*R)$, one first constructs Koszul objects $R\mm\fr$ \cite[Section 3]{bchv1}, the homotopical quotient of $R$ by $\fr$. This object depends on a choice of generators of $\fr$, but the thick subcategory $\Thick_R(R\mm\fr)$ generated by $R\mm\fr$ does not. Consider then the specialization closed set $\cV(\fp)=\{\fq\in\Spec^h(\pi_*R)\mid \fq \supseteq \fp\}$. The inclusion of the localizing subcategory $\Loc_R(R\mm\fr\mid \fr \in\cV(\fp) )$ into $\Mod_R$ has a colimit preserving right adjoint $\Gamma_{\cV(\fp)}$. Moreover, there is a localization functor $L_{\cZ(\fp)}$  on $\Mod_R$ characterized by the property $\pi_*L_{\cZ(\fp)}M\cong(\pi_*M)_{\fp}$ for all $M\in\Mod_R$.

\begin{defn}
	For $\fp\in\Spec^h(\pi_*R)$ we define the local cohomology functor at $\fp$ as 
	\[
	\Gamma_{\fp}=\Gamma_{\cV(\fp)}L_{\cZ(\fp)}.
	\]
The functor $\Gamma_{\fp}$ admits a right adjoint, the local homology functor $\Lambda^{\fp}$ at $\fp$.
\end{defn}

\begin{rem}
	The functor $\Gamma_{\fp}$ is smashing, that is, $\Gamma_{\fp}M\simeq\Gamma_{\fp}R\otimes M$ for all $R$-modules $M$. In what follows we will denote the essential image of $\Gamma_{\fp}$ by $\Gamma_{\fp}\Mod_R$, and likewise for $\Lambda^{\fp}$.
\end{rem}
We refer the reader to the aforementioned references for a more detailed construction and further properties. These functors give rise to a good notion of support and cosupport for $R$-modules as follows. 

\begin{defn}
  Let $M$ be an $R$-module, then we define the support and cosupport of $M$ as
\[
\supp_R(M)  = \{ \frak p \in \Spec^h(\pi_*R) \mid \Gamma_{\frak p}M \not \simeq 0\}, \quad \text \quad \cosupp_R(M) = \{ \frak p \in \Spec^h(\pi_*R) \mid \Lambda^{\frak p}M  \not \simeq 0\}.
\]
\end{defn}

The next theorem provides some evidence for the usefulness of these functors, see \cite[Theorem 7.2]{bik11} and \cite[Remark 8.8]{bik12}.

\begin{thm}[Benson--Iyengar--Krause]\label{thm:local_global}
If $R$ is Noetherian, then the local-to-global principle holds for $R$, i.e., for every $M \in \Mod_R$ we have
\[
\Loc_R(M) = \Loc_R(\Gamma_{\fp}M \mid \fp \in \Spec^h(\pi_*R)), \quad \text{} \quad \Coloc_R(M) = \Coloc_R(\Lambda^{\fp}M \mid \fp \in \Spec^h(\pi_*R)).
\]
In particular, $M$ is trivial if and only if $\supp_R(M) = \varnothing$ if and only if $\cosupp_R(M) = \varnothing$.
\end{thm}

We will use the fact that both support and cosupport detect zero objects implicitly throughout the remainder of this paper. 

Support theory is used to give a classification of the localizing subcategories of $\Mod_R$. We note that there is a pair of natural morphisms

\begin{equation}\label{eq:locsubcat}
\begin{Bmatrix}
\text{Localizing subcategories} \\
\text{of $\Mod_{R}$} 
\end{Bmatrix} 
\xymatrix@C=2pc{ \ar@<0.75ex>[r]^{\supp_R} &  \ar@<0.75ex>[l]^{\supp^{-1}_R}}
\begin{Bmatrix}
\text{Subsets of $\Spec^h(\pi_*R)$} 
\end{Bmatrix},
\end{equation}
where for a localizing subcategory $\mathcal{S} \subseteq \Mod_R$ and a subset $\mathcal{U} \subseteq \Spec^h(\pi_*R)$ we set 
\[
\supp_R(\mathcal{S}) = \bigcup_{M \in \mathcal{S}} \supp_R(M) \quad \text{ and } \quad \supp_R^{-1}(\mathcal{U}) = \{ M \in \Mod_R \mid \supp_R(M) \subseteq \mathcal{U}\}.
\]
Benson, Iyengar, and Krause \cite{bik11} exhibit conditions to ensure that the maps in \eqref{eq:locsubcat} are bijections. To the state their definition, recall that a localizing subcategory $\cD$ of $\Mod_R$ is said to be minimal if $\cD$ has no proper localizing subcategories.

\begin{defn}
Let $R$ be a Noetherian commutative ring spectrum. We say that $\Mod_R$ is (canonically) stratified if for each $\fp \in \Spec^h(\pi_*R)$, the localizing subcategory $\Gamma_{\fp}\Mod_R$ is minimal. 
\end{defn}

\begin{rem}
  This is slightly different from the definition for triangulated categories given in \cite{bik11}, where they require that the local to global principle holds and that $\Gamma_{\fp}\Mod_R$ is minimal or non-zero. As seen in \Cref{thm:local_global} the local to global principle always holds for $\Mod_R$ when $R$ is Noetherian, while in \cite[Proposition 2.13(1)]{bchv1} we show that $\Gamma_{\fp}\Mod_R$ is always non-zero. Since we will only consider the canonical action of $\pi_*R$ on $\Mod_R$, we will from now on omit the adjective canonical when referring to stratification.
\end{rem}

The next result demonstrates the usefulness of this definition, see \cite[Theorem 4.2]{bik11} for the proof.
\begin{thm}[Benson--Iyengar--Krause]
  If $\Mod_R$ is stratified, then the maps \eqref{eq:locsubcat} are inclusion preserving bijections. 
\end{thm}

As a consequence, the telescope conjecture hold for $\Mod_R$, see \cite[Theorem 6.3]{bik11} and \cite[Theorem 11.13]{bik_finitegroups}. 

\begin{thm}[Benson--Iyengar--Krause]\label{thm:tel_conj}
	Suppose $\Mod_R$ is stratified by $\pi_*R$. Let $\cal{U}$ be a localizing subcategory of $\Mod_{R}$. Then the following conditions are equivalent. 
	\begin{enumerate}
		\item The localizing subcategory $\cal{U}$ is smashing, i.e., the associated localizing endofunctor on $\Mod_R$ preserves colimits.
		\item The localizing subcategory $\cal{U}$ is generated by compact objects in $\Mod_R$. 
		\item The support of $\cal{U}$ is specialization closed, i.e, it is a union of Zariski closed subsets of $\Spec^h(\pi_*R)$. 
	\end{enumerate}
\end{thm}
\subsection{Base-change and conservative maps of ring spectra}

A remarkable result proven by Chouinard~\cite{chouinard} is that the projectivity of a module over the group ring of a finite group $G$ can be tested at all elementary abelian $p$-subgroup of $G$ combined. This can be interpreted as a way to detect trivial objects when working in the stable module category of $G$. In this section, we study homotopical analogues of this result in the setting of commutative ring spectra. We do so by expressing the detection of trivial objects as a conservativity property of certain functors that we introduce hereunder.

Let $f\colon R \to S$ be a morphism of commutative ring spectra and view $S$ as an $R$-module via $f$. Forgetting along $f$ induces a restriction functor $f_* = \Res\colon \Mod_S \to \Mod_R$ which admits both a left adjoint $f^*$ and a right adjoint $f^!$, given by induction (or extension of scalars) along $f$,
\[
\xymatrix{f^* = \Ind = S \otimes_R (-)\colon \Mod_R \ar[r] & \Mod_S,}
\]
and coinduction, defined as 
\[
\xymatrix{f^! = \Coind = \Hom_R(S,-)\colon \Mod_R \ar[r] & \Mod_S.}
\] 
Note that restriction detects equivalences, i.e., a map $\alpha \colon N_1 \to N_2$ between $S$-modules is an equivalence if and only if $f_*(\alpha)$ is. It follows that $f^*$ or $f^!$ are conservative if and only if the composites $f_*f^*$ or $f_*f^!$ are, respectively. We capture this property by introducing the following terminology:

\begin{defn}
A map $f\colon R \to S$ of commutative ring spectra is called conservative (resp.~coconservative) if the associated induction functor $f^*\colon \Mod_R \to \Mod_S$ (resp.~coinduction $f^!$) is conservative. If a map is both conservative and coconservative, it will be called biconservative. 
\end{defn}

In \cite{bchv1} we studied the base-change properties of support and cosupport along a morphism of ring spectra $f \colon R \to S$. We single out the following results, where we let 
\[
\xymatrix{\res_f \colon \Spec^h(\pi_*S) \ar[r] & \Spec^h(\pi_*R)} 
\]
denote the map induced by $f$ on $\Spec^h(\pi_*(-))$.  The first two items of the next result are proven in \cite[Corollary 3.9 and Proposition 3.12]{bchv1}, while the third one strengthens \cite[Lemma 3.11]{bchv1}.

\begin{thm}\label{thm:suppbasechange}
Let $f\colon R \to S$ be a map of commutative ring spectra, $M \in \Mod_R$ and $N \in \Mod_S$, then we have:
  \begin{enumerate}
    \item $\res_f\supp_S(f^* M) \subseteq \supp_R(M)$, and equality holds if $f$ is conservative.
    \item $\supp_R(f_* N) = \res_f\supp_S(N)$ and $\cosupp_R(f_*N) = \res_f\cosupp_S(N)$.
    \item $\res_f \cosupp_S(f^! M) \subseteq \cosupp_R(M)$, and equality holds if $f$ is coconservative. 
  \end{enumerate}
\end{thm}
\begin{proof}
In light of the references provided above, it remains to verify the last claim. The inclusion ``$\subseteq$'' holds unconditionally, so it suffices to establish ``$\supseteq$'' when $f$ is coconservative. Let $M \in \Mod_R$ and $\fp \in \cosupp_R(M)$. By adjunction, there are equivalences of $S$-modules
\[
0 \not\simeq f^!\Lambda^{\fp}M \simeq f^!\Hom_R(R,\Lambda^{\fp}M) \simeq f^!\Hom_R(\Gamma_{\fp}R,M) \simeq \Hom_S(f^*\Gamma_{\fp}R,f^!M).
\]
The last equivalence can be found, for example, in Proposition 2.15 of \cite{bds_wirth}, see their Equation (2.19). Let $\cV$ and $\cW$ be specialization closed subsets of $\Spec^h(\pi_*R)$ such that $\cV \setminus \cW = \{\fp\}$ and write $\widetilde{\cV} = \res_f^{-1}(\cV)$ and $\widetilde{\cW} = \res_f^{-1}(\cW)$; note in particular that $\widetilde{\cV} \setminus \widetilde{\cW} = \res_f^{-1}(\fp)$. It then follows from Corollary 3.8 in \cite{bchv1} and adjunction that
\[
\Hom_S(f^*\Gamma_{\fp}R,f^!M) \simeq \Hom_S(L_{\widetilde{\cW}}\Gamma_{\widetilde{\cV}}f^*R,f^!M) \simeq  \Hom_S(S,\Lambda^{\widetilde{\cV}}\Delta^{\widetilde{\cW}} f^!M) \simeq \Lambda^{\widetilde{\cV}}\Delta^{\widetilde{\cW}} f^!M.
\]
By Lemma 3.4 of \cite{bchv1}, the cosupport of this expression can thus be computed to be
\[
\varnothing \neq \cosupp_S(f^!\Lambda^{\fp}M) = \cosupp_S(\Lambda^{\widetilde{\cV}}\Delta^{\widetilde{\cW}} f^!M) = \res_f^{-1}(\fp) \cap \cosupp_S(f^!M).
\]
This implies that $\fp \in \res_f\cosupp_S(f^!M)$ as desired.
\end{proof}

\begin{lem}\label{lem:ressurj}
Fix a prime ideal $\fp \in \Spec^h(\pi_*R)$. If $f\colon R \to S$ is a map of commutative ring spectra such that induction $f^*\colon \Mod_R \to \Mod_S$ considered as a functor on $\Gamma_{\fp}\Mod_R$ is conservative (resp. coinduction as a functor on $\Lambda^{\fp}\Mod_R$ is conservative), then $\fp$ is in the image of $\res_f\colon \Spec^h(\pi_*S) \to \Spec^h(\pi_*R)$.
\end{lem}
\begin{proof}
Using \Cref{thm:suppbasechange}(1) and (2), we get 
\[
\supp_R(f_*f^* \Gamma_{\fp}R) = \res_f \supp_S (f^*\Gamma_{\fp}R) \subseteq \supp_R(\Gamma_{\fp}R) = \{\fp\}.
\]
Now suppose that $\fp$ is not in the image of $\res_f$, then $f_*f^* \Gamma_{\fp}R \simeq 0$. Therefore, $f^* \Gamma_{\fp}R \simeq 0$, so $f^* \colon \Gamma_{\fp}\Mod_R \to \Mod_S$ is not conservative. The version for coinduction follows similarly from \Cref{thm:suppbasechange}(3).
\end{proof}

Under the assumption that $\Gamma_{\fp}\Mod_R$ is a minimal localizing subcategory of $\Mod_R$, the converse holds as well, as we will show after some preparation in \Cref{lem:chouinardiff} below.

\begin{lem}\label{lem:p-tensor_hom_(co)supp}
	Let $\fp\in \Spec^h(\pi_*R)$ such that $\Gamma_{\fp}\Mod_R$ is minimal. Then, for all $M\in \Gamma_{\fp}\Mod_R$ and all $N\in \Mod_R$,
	\[
	\supp_R(M\otimes N)=\supp_R (M)\cap \supp_R (N).
	\]
Dually, for all $M\in \Mod_R$ and all $N\in \Lambda^{\fp}\Mod_R$
	\[
	\cosupp_R(\Hom_R(M,N))=\supp_R (M)\cap \cosupp_R (N).
	\]
\end{lem}
\begin{proof}
	Since $M\in\Gamma_{\fp}\Mod_R$, we have
	\[
	\supp_R(M\otimes N)\subseteq \supp_R (M)\cap \supp_R (N)\subseteq \{\fp\}.
	\]
	Under the minimality assumption on $\Gamma_{\fp}\Mod_R$, the same argument as in the proof of \cite[Theorem 7.3]{bik11} implies that if $\supp_R (M)\cap \supp_R (N)=\{\fp\}$ then $\supp_R(M\otimes N)=\{\fp\}$, which proves the first claim.
	
	Likewise, the second equality follows from the proof of \cite[Theorem 9.5]{bik12}.
\end{proof}

The next result is a local analogue of \cite[Proposition 3.14]{bchv1}.

\begin{prop}\label{prop:abstractsubgroupthm}
Let $f\colon R \to S$ be a morphism of Noetherian commutative ring spectra and suppose that $\fp\in\Spec^h(\pi_*R)$ is such that $\Gamma_{\fp}\Mod_R$ is minimal. Then, for all $M \in \Gamma_{\fp}\Mod_R$, 
\[
\supp_{S}(f^* M) = \res_f^{-1}\supp_{R}(M), \quad \text{while} \quad
 \cosupp_S(f^! N) = \res_f^{-1}\cosupp_R (N),
\]
for all $N\in\Lambda^{\fp}\Mod_R$.
\end{prop}
\begin{proof}
Let $\fq\in \Spec^h(\pi_*S)$ with $\fp = \res_f(\fq)$. First, we can use the projection formula \cite[Lemma 2.2]{bchv1} to write 
\[
	f_*\Gamma_{\fq}f^* M \simeq f_*(\Gamma_{\fq}S \otimes_S f^* M) \simeq f_*(\Gamma_{\fq}S)\otimes_RM. 
	\]
By \Cref{thm:suppbasechange} and \cite[Proposition 2.13(1)]{bchv1} we have
\begin{equation}\label{eq:supp_formula}
	\supp_{R}(f_* \Gamma_{\fq}S) = \res_f\supp_S(\Gamma_{\fq}S)  = \{\res_f(\fq)\}= \{ \fp \}.
\end{equation}
Since $M\in\Gamma_{\fp}\Mod_R$, \cref{lem:p-tensor_hom_(co)supp} applies in this situation: 
\[
\supp_R(f_*(\Gamma_{\fq}S)\otimes_RM) = \supp_R(f_*\Gamma_{\fq}S) \cap \supp_R(M)\subseteq \{\fp \}.
\]
It follows that $f_*\Gamma_{\fq}f^* M\not\simeq0$ if and only if $\fp \in \supp_R(M)$. But $f_*$ is conservative, so we conclude that $\fq \in \supp_S(f^* M)$ if and only if $\fp \in \supp_R(M)$.

For the second claim, note that by adjunction there are equivalences of $R$-modules
  \begin{align*}
    \Lambda^{\fq}f^! N & \simeq \Hom_S(S,\Lambda^{\fq}f^! N) \\
    & \simeq \Hom_S(\Gamma_{\fq}S,f^! N) \\
    & \simeq \Hom_R(f_*\Gamma_{\fq}S,N).
  \end{align*}
But we are assuming that $N\in\Lambda^{\fp}\Mod_R$, so \cref{lem:p-tensor_hom_(co)supp} yields
\[
\cosupp_R(\Hom_R(f_*\Gamma_{\fq}S,N))=\supp_R(f_*\Gamma_{\fq}S) \cap \cosupp_R(N)= \{\fp\} \cap \cosupp_R(N),
\]
where we used \eqref{eq:supp_formula}. We conclude that $\fq\in\cosupp_S(f^!N)$ if and only if $ \fp \in\cosupp_R (N)$.
\end{proof}

\begin{lem}\label{lem:chouinardiff}
Fix a prime ideal $\fp \in \Spec^h(\pi_*R)$, let $f\colon R \to S$ be a map of Noetherian commutative ring spectra, and assume that $\Gamma_{\fp}\Mod_R$ is minimal. Then the following three conditions are equivalent:
	\begin{enumerate}
		\item Induction $f^*\colon \Mod_R \to \Mod_S$ restricted to $\Gamma_{\fp}\Mod_R$ is conservative.
		\item The prime ideal $\fp$ is in the image of 
\[
\xymatrix{\res_f\colon \Spec^h(\pi_*S) \ar[r] & \Spec^h(\pi_*R).}
\] 
		\item Coinduction $f^!\colon \Mod_R \to \Mod_S$ restricted to $\Lambda^{\fp}\Mod_R$ is conservative.
	\end{enumerate}
\end{lem}
\begin{proof}
\Cref{lem:ressurj} gives one implication, so consider $M \in \Gamma_{\fp}\Mod_R$ with $f^*(M) \simeq 0$. By \cref{prop:abstractsubgroupthm}, this implies that 
\[
\varnothing = \supp_S(f^*M) = \res_f^{-1}\supp_R(M).
\]
Since $\fp$ is in the image of $\res$ and $\supp_R(M) \subseteq \{\fp\}$, it follows that $\supp_R(M) = \varnothing$, hence $M \simeq 0$ by \Cref{thm:local_global}. The dual equivalence between (2) and (3) is proven similarly. 
\end{proof}

\begin{rem}
Note that the assumption in the previous lemmas are slightly asymmetric: even though we deduce statements about induction and coinduction, we only assume that $\Gamma_{\fp}\Mod_R$ (and not $\Lambda^{\fp}\Mod_R$) is minimal.
\end{rem}

We can combine the previous two local lemmas in a global statement, which says that conservativity for a morphism $f$ of ring spectra is a strengthened form of surjectivity of $\res_f$. 

\begin{prop}\label{prop:reschouinard}
Suppose $f\colon R \to S$ is a map of commutative ring spectra and consider the following two conditions:
	\begin{enumerate}
		\item $f$ is conservative (resp.~coconservative).
		\item $\res_f\colon \Spec^h(\pi_*S) \to \Spec^h(\pi_*R)$ is surjective.  
	\end{enumerate}
Then (1) implies (2), and the reverse implication holds provided $R$ is Noetherian and $\Mod_R$ is stratified. 
\end{prop}
\begin{proof}
We will only prove the claim about conservativity of $f$, leaving the modifications for the dual proof to the reader. If $f$ is conservative, then $f^*\colon \Gamma_{\fp}\Mod_R \subseteq \Mod_R \to \Mod_S$ is conservative for all $\fp \in \Spec^h(\pi_*R)$, so $\Spec^h(\pi_*R)$ is in the image of $\res_f$ by \Cref{lem:ressurj}. This establishes the implication $(1) \implies (2)$.
 
Conversely, assume that $\Mod_R$ is stratified and that $\res_f$ is surjective. Let $\fp\in\Spec^h(\pi_*R)$ and $M\in\Mod_R$ with $f^*(M)\simeq0$. This implies that 
\[
f^*(\Gamma_{\fp}M) \simeq f^*(\Gamma_{\fp}R) \otimes_S f^*(M) \simeq 0.
\] 
Since $\res_f$ is surjective, \Cref{lem:chouinardiff} yields $\Gamma_{\fp}M\simeq0$. But $\fp$ was arbitrary, so it follows by the local-to-global principle \cref{thm:local_global} that $M\simeq0$.
\end{proof}

\begin{rem}
The base-change formula in \cref{thm:suppbasechange} applied to $\supp_R(f_*S)$ yields
\[
	\supp_R(f_*S)=\res_f(\supp_SS)=\res_f(\Spec^h(\pi_*S)).
\]
Thus, $\res_f$ being surjective is the same as having $\supp_R(f_*S)=\Spec^h(\pi_*R)$.
\end{rem}

\begin{lem}\label{lemma:chouinard-transitivity}
Let $f\colon R\rightarrow S$ and $g\colon S\rightarrow T$. If both are conservative (resp.~coconservative) then so is their composite. Conversely, if $g\circ f$ is conservative (resp.~coconservative) then so is $f$; if $f^*$ is additionally essentially surjective, then $g$ is also conservative (resp.~coconservative).  
\end{lem}
\begin{proof}
We verify the last claim about conservativity and leave the proof of the remaining statements to the reader. Consider $N \in \Mod_S$ with $g^*(N) \simeq 0$. By assumption, there exists $M \in \Mod_S$ such that $f^*(M) \simeq N$, hence $(gf)^*(M) \simeq g^*(f^*(M))\simeq 0$. Since $(gf)^*$ is conservative, $M \simeq 0$, so $N \simeq 0$. 
\end{proof}

\begin{lem}\label{lem:chouinardloc}
If $R \in \Loc_R(S)$, then $f$ is biconservative. In particular 
this holds if $f\colon R\rightarrow S$ admits an $R$-module retract.
\end{lem}
\begin{proof}
Indeed, given $M \in \Mod_R$ such that $f^*(M) \simeq S \otimes_R M \simeq 0$, it follows that $N \otimes_R M \simeq 0$ for any $N \in \Loc_R(S)$. In particular, $M \simeq R \otimes_R M \simeq 0$. The argument for coconservativity is similar. Finally, if $f$ admits an $R$-module retract, then $R \in \Loc_R(S)$.
\end{proof}

We end this section with a formal base-change property for (co)conservativity of maps that will turn out to be useful later.

\begin{lem}\label{lem:chouinardbasechange}
Given a pushout diagram of commutative ring spectra
\[
\xymatrix{R \ar[r]^{i} \ar[d]_f & S \ar[d]^g \\
A \ar[r]^j & B}
\]
such that $f$ is conservative (resp.~coconservative), then $g$ is conservative (resp.~coconservative) as well. 
\end{lem}
\begin{proof}
Both statements are consequences of the Beck--Chevalley condition and its dual, which we briefly recall for the convenience of the reader. The natural equivalence $i_*g_* \simeq f_*j_*$ has a mate $\zeta\colon f^*i_* \to j_*g^*$ given by the composite 
\[
\xymatrix{f^*i_* \ar[r]^-{\eta_g} & f^*i_*g_*g^* \simeq f^*f_*j_*g^* \ar[r]^-{\epsilon_f} & j_*g^*,}
\]
where $\eta_g$ is the unit of the adjunction $(g^* \dashv g_*)$ and  $\epsilon_f$ is the counit of the adjunction $(f^* \dashv f_*)$. In order to show that $\zeta$ is an equivalence, it suffices to evaluate it on $S$ and restrict to $\Mod_R$. The claim then follows from the equivalence $A \otimes_R S \simeq B$ of $R$-modules. Passing to right adjoints yields natural equivalences $i^!f_* \simeq g_*j^!$ and, by symmetry, $j_*g^! \simeq f^!i_*$. 

Suppose now that $N \in \Mod_S$ such that $g^*(N) \simeq 0$, then there are equivalences of $A$-modules $0\simeq j_*g^*(N) \simeq f^*i_*(N)$. Therefore, $i_*(N) \simeq 0$ by assumption on $f^*$ and thus $N \simeq 0$. Similarly, if $f^!$ is conservative, then $0 \simeq j_*g^!(N) \simeq f^!i_*(N)$ implies $N \simeq 0$.
\end{proof}

\section{Quillen lifting and stratification}\label{sec:quillen}
In this section we study stratification in the sense of \cite{bik11} for structured ring spectra. We review and extend our previous work \cite{bchv1}, showing how to descend stratification along a morphism of ring spectra. To do so, we introduce the notion of simple Quillen lifting, which is a stronger form of the Quillen lifting property considered in \cite{bchv1}.

\subsection{Simple Quillen lifting}\label{ssec:simple}

In \cite{bchv1} we introduced the concept of Quillen lifting for a morphism of ring spectra $f\colon R \to S$, which we recall below. 

\begin{defn}\sloppy\label{defn:qlifting}
A morphism of Noetherian commutative ring spectra $f\colon R \to S$ is said to satisfy Quillen lifting if for any two modules $M,N \in \Mod_R$ such that there is 
\[
\medmuskip=2mu \fp \in \res_f(\supp_S(f^* M)) \cap \res_f(\cosupp_S(f^! N)),
\]
there exists a prime ideal $\cQ_f(\fp):=\fq \in \res_f^{-1}(\fp)$ with $\medmuskip=2mu \fq \in \supp_S(f^* M) \cap \cosupp_S(f^! N)$. In this case, we will refer to $\cQ(\fp) = \cQ_f(\fp)$ as a Quillen lift of $\fp$ along $f$.
\end{defn}

We will introduce a variant called simple Quillen lifting that is satisfied in many examples and implies Quillen lifting, and establish several of its basic properties.

\begin{defn}
A morphism $f\colon R \to S$ of commutative ring spectra is said to satisfy simple Quillen lifting if for any $M \in \Mod_R$:
\[
\res_f^{-1}\res_f\supp_S (f^* M) = \supp_S (f^* M) \quad \text{and} \quad \res_f^{-1}\res_f\cosupp_S(f^! M) = \cosupp_R(f^! M).
\]
Note that the inclusions $\supseteq$ hold unconditionally.
\end{defn}

\begin{lem}\label{lem:simplequillenlifting}
Simple Quillen lifting of a morphism $f\colon R \to S$ implies Quillen lifting.
\end{lem}
\begin{proof}
Let $\fp$ be in $\res_f(\supp_S(f^* M)) \cap \res_f(\cosupp_S(f^! N))$, so that $\res_f^{-1}(\fp)$ is non-empty. It then follows from simple Quillen lifting that we have inclusions
\begin{align*}
\res_f^{-1}(\fp) & \subseteq \res_f^{-1}\res_f(\supp_S(f^* M)) \cap \res_f^{-1}\res_f(\cosupp_S(f^! N)) \\
& = \supp_S(f^* M) \cap \cosupp_S(f^! N), \end{align*}
so that $f$ satisfies Quillen lifting. 
\end{proof}

The following observation is immediate from the definition of simple Quillen lifting.

\begin{lem}\label{lem:injective}
If the restriction map $\res_f\colon \Spec^h(\pi_*S) \to \Spec^h(\pi_*R)$ is injective, then $f\colon R\to S$ satisfies simple Quillen lifting. 
\end{lem}
\begin{proof}
If $\res_f$ is injective, then $\res_f^{-1}\res_f (\cU) = \cU$ for any subset $\cU\subseteq\Spec^h(\pi_*S)$. 
\end{proof}

\begin{defn}
	We say that a map of (discrete) rings $g \colon A \to B$ is an $\mathcal{N}$-monomorphism if the kernel of $g$ is nilpotent, and that $g$ is a $\mathcal{N}$-epimorphism if for every $b \in B$ there exists an $\ell$ such that $b^{\ell} \in \im g$. A morphism of ring spectra $f\colon R \to S$ is an $\mathcal{N}$-monomorphism (resp. $\mathcal{N}$-epimorphism) if $f_* \colon \pi_*R \to \pi_*S$ is. 
\end{defn}

\begin{lem}\label{lem:niso}
If $f \colon R \to S$ is an $\mathcal{N}$-epimorphism, then $\res_f \colon \Spec^h(\pi_*S) \to \Spec^h(\pi_*R)$ is injective. 
\end{lem}
\begin{proof}
  The argument given in \cite[Proposition 3.24]{mnn} can easily be adjusted to prove this. Namely, factor $f_*$ as the composite $\pi_*R \to \pi_*R/\ker(f_*) \to \pi_*S$. The first map is surjective and so induces an injection on $\Spec^h$. The fact that $\Spec^h(\pi_*S) \to \Spec^h(\pi_*R/\ker(f_*))$ is injective is shown in \cite[Proposition 3.24]{mnn}, and we are done.  
\end{proof}

From \Cref{lem:injective,lem:niso} we deduce the following. 
\begin{cor}\label{lem:nilpotentCoker}
	If $f \colon R \to S$ is an $\mathcal{N}$-epimorphism, then $f$ satisfies simple Quillen lifting. 
\end{cor}

\begin{lem}\label{lem:stratsql}
Let $f\colon R \to S$ be a morphism of commutative ring spectra. If $\Mod_R$ is stratified, then $f$ satisfies simple Quillen lifting and hence Quillen lifting.
\end{lem}
\begin{proof}
Let $M \in \Mod_R$. We have
\[
\res_f^{-1}\res_f(\supp_S(f^*M)) \subseteq \res_f^{-1}\supp_R(M) = \supp_S(f^*M),
\]
where the second equality uses Proposition 3.14 from \cite{bchv1}, which applies because $\Mod_R$ is stratified by assumption. The reverse inclusion holds unconditionally. Similarly, we deduce that
\[
\res_f^{-1}\res_f(\cosupp_S(f^!M)) = \cosupp_S(f^!M).
\]
This shows that $f$ satisfies simple Quillen lifting and hence also Quillen lifting by \Cref{lem:simplequillenlifting}. 
\end{proof}

\begin{rem}
It is not the case that a morphism $f\colon R \to S$ which satisfies Quillen lifting and such that $\Mod_R$ and $\Mod_S$ are stratified has to be conservative. Indeed, choose a surjective morphism of discrete Noetherian commutative rings $A \to B$ which gives a morphism of commutative Noetherian ring spectra $f \colon HA \to HB$. Since $\Mod_{HA} \simeq \cD(A)$, the unbounded derived category of $A$, and $\Mod_{HB} \simeq \cD(B)$ (see \cite[Remark 7.1.1.16]{ha}), these are both stratified by Neeman's result \cite{neemanchromtower}. Moreover, because $A \to B$ is surjective, simple Quillen lifting is satisfied by \Cref{lem:injective}. However, by \Cref{prop:reschouinard} induction along $f$ is conservative if and only if the induced map $\res_f \colon \Spec(B) \to \Spec(A)$ is surjective, which does not hold in general.
\end{rem}

Finally, we establish the following partial transitivity properties. 

\begin{prop}\label{prop:transitivity}
Let $R \xrightarrow{f} S \xrightarrow{g} T$ be morphisms of commutative ring spectra and assume that $g$ is biconservative. If both $f$ and $g$ satisfy simple Quillen lifting, then so does $gf$. 
\end{prop}
\begin{proof}
We will first verify the claim for induced modules. To this end, let $M \in \Mod_R$; we then have equalities:
\begin{align*}
\res_{gf}^{-1}\res_{gf}\supp_T((gf)^*M) & = \res_{gf}^{-1}\res_{f}\res_{g}\supp_T(g^*f^*M) & & \\
& = \res_{g}^{-1}\res_{f}^{-1}\res_{f}\supp_S(f^*M) & & \text{\Cref{thm:suppbasechange}(1)} \\
& = \res_{g}^{-1}\supp_S(f^*M) & & \text{simple Quillen lifting for }f \\
& = \res_{g}^{-1}\res_{g}\supp_T(g^*f^*M) & & \text{\Cref{thm:suppbasechange}(1)} \\
& = \supp_T((gf)^*M) & & \text{simple Quillen lifting for }g,
\end{align*}
where we used that $g$ is conservative in the penultimate equality. The argument is similar for coinduction, using \Cref{thm:suppbasechange}(3). 
\end{proof}

We can also prove the following converse to the last result. 

\begin{prop}
Let $R \xrightarrow{f} S \xrightarrow{g} T$ be morphisms of commutative ring spectra such that $g$ is biconservative. If $gf$ satisfies simple Quillen lifting, then so does $f$.
\end{prop}
\begin{proof}
Simple Quillen lifting for $gf$ and the same argument as in \Cref{prop:transitivity} gives
\begin{align*}
\res_{g}^{-1}\res_{f}^{-1}\res_{f}\supp_S(f^*M) & = \res_{gf}^{-1}\res_{gf}\supp_T((gf)^*M) \\
& = \supp_T((gf)^*M) \\
&  \subseteq \res_{g}^{-1}\res_{g} \supp_T(g^*f^*M) \\
& = \res_{g}^{-1}\supp_S(f^*M).
\end{align*}
By \Cref{prop:reschouinard}, $\res_{g}$ is surjective, so applying this map to the inclusion above yields an inclusion
\[
\res_{f}^{-1}\res_{f}\supp_S(f^*M) \subseteq \supp_S(f^*M),
\]
while the reverse inclusion always holds. The same argument works for coinduction. 
\end{proof}

\subsection{Descent for stratifications}\label{ssec:stratdescent}

In this section we isolate a condition for descending stratification along a morphism of Noetherian commutative ring spectra $f \colon R \to S$.
\begin{prop}\label{prop:minconditions}
Let $f\colon R \to S$ be a morphism of Noetherian commutative ring spectra. For a fixed prime ideal $\fp \in \Spec^h(\pi_*R)$, the following conditions are equivalent:
	\begin{enumerate}
		\item $\Gamma_{\fp}\Mod_R$ is minimal and $\fp \in \im(\res\colon \Spec^h(\pi_*S) \to \Spec^h(\pi_*R))$. 
		\item If $M, N \in \Gamma_{\fp}\Mod_R$ are non-zero modules, then $\Hom_S(f^* M, f^! N) \not\simeq 0$.
	\end{enumerate}
If we assume additionally that $\Mod_S$ is stratified, then both conditions are equivalent to:
	\begin{enumerate}
		\item[(3)] If $M, N \in \Gamma_{\fp}\Mod_R$ are non-zero modules, then 
		\[
		\supp_S(f^* M) \cap \cosupp_S(f^! N) \neq \varnothing.
		\]
	\end{enumerate}
\end{prop}
\begin{proof}
We first establish the equivalence between Conditions (1) and (2). By \cite[Lemma 4.1]{bik11}, $\Gamma_{\fp}\Mod_R$ is a minimal localizing subcategory of $\Mod_R$ if and only if $\Hom_R(M,N) \not\simeq 0$ for all non-zero $M,N \in \Gamma_p\Mod_R$. Assume (1), then \Cref{lem:chouinardiff} shows that $f_*f^*(M) \simeq 0$ if and only if $M \in \Gamma_{\fp}\Mod_R$ is zero. This implies that $\supp_R( f_*f^*M) = \supp_R (M)$ for all $M \in \Gamma_{\fp}\Mod_R$. Therefore, for non-zero $M,N \in \Gamma_{\fp}\Mod_R$, we have
\[
0 \not\simeq \Hom_R(f_*f^* M,N) \simeq \Hom_S(f^* M, f^! N). 
\]
Conversely, (2) says $0 \not\simeq \Hom_R(f_*f^* M,N)$ for non-zero $M,N \in \Gamma_{\fp}\Mod_R$, hence $0 \not\simeq \Hom_R(M,N)$ as well as, since $f_*f^*M \in \Loc_R(M)$. Therefore, $\Gamma_{\fp}\Mod_R$ is a minimal localizing subcategory of $\Mod_R$. Furthermore, (2) implies that $f^*$ and $f^!$ are conservative when restricted to $\Gamma_{\fp}\Mod_R$. \Cref{lem:ressurj} then shows $\fp \in \im(\res_f\colon \Spec^h(\pi_*S) \to \Spec^h(\pi_*R))$.

The equivalence of Conditions (2) and (3) is a direct consequence of the formula 
\[
\cosupp_S(\Hom_S(X,Y)) = \supp_S(X) \cap \cosupp_S(Y)
\]
for $X,Y \in \Mod_S$, which by \cite[Theorem 9.5(2)]{bik12} holds whenever $\Mod_S$ is stratified. 
\end{proof}

\begin{cor}\label{cor:stratiff}
Let $f\colon R \to S$ be a conservative (or coconservative) morphism of Noetherian commutative ring spectra and assume $\Mod_S$ is stratified, then $\Mod_R$ is stratified if and only if the following condition holds:
	\begin{enumerate}
		\item[$(\ast)$] For any homogeneous prime ideal $\fp \in \Spec^h(\pi_*R)$ and non-zero modules $M, N \in \Gamma_{\fp}\Mod_R$ we have $\supp_S(f^*M) \cap \cosupp_S(f^! N) \neq \varnothing$.
	\end{enumerate}
\end{cor}

\begin{proof}
Assume first that $(\ast)$ holds. In order to show that $\Mod_R$ is stratified, we must show that the localizing subcategories $\Gamma_{\fp}\Mod_R$ are minimal for all $\fp \in \Spec^h(\pi_*R)$, which follows by the implication $(3) \implies (1)$ of \Cref{prop:minconditions}. Conversely, suppose $\Mod_R$ is stratified. Since $f$  is conservative (or coconservative), $\res_f\colon \Spec^h(\pi_*S) \to \Spec^h(\pi_*R)$ is surjective by \Cref{prop:reschouinard}, so $(\ast)$ follows from $(1) \implies (3)$ of \Cref{prop:minconditions}. 
\end{proof}
From the previous corollary, we can then deduce that Quillen lifting is not only sufficient, but in fact also necessary for descent of stratification along a conservative and coconservative morphism of Noetherian ring spectra. 

\begin{thm}\label{thm:stratification}
Let $f\colon R \to S$ be a biconservative morphism of Noetherian commutative ring spectra and assume $\Mod_S$ is stratified, then the following are equivalent:
	\begin{enumerate}
		\item $\Mod_R$ is stratified.
		\item $f$ satisfies Condition $(\ast)$.
		\item $f$ satisfies Quillen lifting. 
		\item $f$ satisfies simple Quillen lifting.
	\end{enumerate}
\end{thm}
\begin{proof}
The equivalence of (1) and (2) is the content of \Cref{cor:stratiff}. Next, we will prove that Quillen lifting implies Condition $(\ast)$, thereby establishing that $(3) \implies (2)$. Since $f$ is biconservative, \Cref{thm:suppbasechange} implies that for all $M,N \in \Mod_R$ there are equalities
\[
\res_f\supp_S(f^* M) = \supp_R(M) \quad \text{and} \quad \res_f\cosupp_S(f^!N) = \cosupp_R(N).
\]
Given $\fp \in \Spec^h(\pi_*R)$, if $M,N \in \Gamma_{\fp}\Mod_R$ are non-zero modules, then $\supp_R(M) = \{\fp\} = \supp_R(N)$, hence also $\fp \in \cosupp_R(N)$ by \cite[Theorem 4.13]{bik12}. This puts us in the situation of Quillen lifting, hence the Quillen lift of $\fp$ along $f$ gives $\cQ(\fp) \in \supp_S(f^* M) \cap \cosupp_S(f^! N)$. 

Finally, we have that $(1) \implies (4)$ by \Cref{lem:stratsql}, and $(4) \implies (3)$ by \Cref{lem:simplequillenlifting}. \end{proof}

\begin{rem}\label{rem}
If $f$ is finite, then $f^*$ is conservative if and only if $f^!$ is conservative, see \cite[Lemma 3.17]{bchv1}. Further, if there exists a natural equivalence $f^* \simeq f^!$, then $f$ is finite. Indeed, this equivalence implies that $f^!$ preserves colimits, and hence the left adjoint $f_*$ preserves compact objects. This for example is the situation occurring for the stable module category of a finite group in \cite{bik_finitegroups}, where induction and coinduction coincide.
\end{rem}

Finally, we relate \Cref{thm:stratification} to the notion of Avrunin--Scott stratification introduced by Hovey and Palmieri in \cite{hp_tst}. Their Definition 5.5 axiomatizes the Avrunin--Scott theorem in modular representation theory, see \cite{avruninscott}, \cite[Theorem 10.7]{bcr2}, and \cite{bik_finitegroups}. The Avrunin--Scott identities appearing in the next result form a suitable generalization to a morphism $f\colon R \to S$ for which $f^*$ is not necessarily equivalent to $f^!$. 

\begin{cor}\label{cor:as}
Under the assumptions of \Cref{thm:stratification}, the equivalent Conditions (1)--(3) hold if and only if the Avrunin--Scott identities are satisfied for all $M \in \Mod_R$:
\[
\supp_{S}(f^* M) = \res_f^{-1}\supp_{R}(M) \quad \text{and} \quad
\cosupp_S(f^! M) = \res_f^{-1}\cosupp_R(M).
\]
\end{cor}
\begin{proof}
Suppose $\Mod_R$ is stratified, then \cite[Proposition 3.14]{bchv1} implies that the Avrunin--Scott identities are satisfied for all $R$-modules $M$. Conversely, assume that these two identities hold for $M,N \in \Mod_R$. Let $\fp$ be a homogeneous prime ideal with
\[
\fp \in \res_f(\supp_S(f^* M)) \cap \res_f(\cosupp_S(f^! N)) \subseteq \supp_R(M) \cap \cosupp_R(N). 
\]
It follows that
\[
\varnothing \neq \res_f^{-1}(\fp) \subseteq \res_f^{-1}(\supp_R(M)) \cap \res_f^{-1}(\cosupp_R(N)) = \supp_{S}(f^* M) \cap \cosupp_S(f^! M),
\]
where the last equivalence uses the Avrunin--Scott identities for $M$ and $N$, respectively. This verifies Quillen lifting for the morphism $f$, so the equivalent Conditions (1)--(3) of \Cref{thm:stratification} hold.
\end{proof}

\subsection{An abstract criterion for simple Quillen lifting}\label{ssec:criterionsql}

The goal of this subsection is to prove an abstract condition which guarantees simple Quillen lifting for a morphism of Noetherian commutative ring spectra whose target decomposes into a finite product. This result and its proof is modelled on the proof of Quillen lifting in Theorem 5.10 of \cite{bchv1}.

\begin{prop}\label{ssec:criterionsql}
Let $I$ be a finite set and $f\colon R \to S = \prod_{i \in I}S_i$ a morphism of Noetherian commutative ring spectra with $\Mod_{S_i}$ stratified for all $i \in I$. Assume that for all $\fp \in \Spec^h(\pi_*R)$ there exist $i=i(\fp) \in I$ and a prime ideal $\cO(\fp) \in \Spec^h(\pi_*S_i)$ satisfying the following condition:
	\begin{enumerate}
		\item[($\clubsuit$)] For every $j \in I$ and every $\fq \in \Spec(\pi_*S_j)$ with $\fp = \res_f(\fq)$ there exists a biconservative map of commutative ring spectra $\zeta(\fq)\colon S_j \to S_i$ under $R$ such that $\res_{\zeta(\fq)}(\cO(\fp)) = \fq$.
	\end{enumerate}
Then $f$ satisfies simple Quillen lifting. 
\end{prop}

In our applications to commutative ring spectra of cochains on spaces, verifying the existence of a section $\cO\colon \Spec^h(\pi_*R) \to \Spec^h(\pi_*\prod_{i \in I}S_i)$ to the map $\res_f$ satisfying the condition ($\clubsuit$) is based on two ingredients:
\begin{itemize}
	\item A form of Quillen stratification for $\Spec^h(\pi_*R)$ in terms of the images of $(\Spec^h(\pi_*S_i))_{i\in I}$.
	\item The realizability of the associated transition maps on homotopy groups to morphisms $\zeta$ of commutative ring spectra, as in ($\clubsuit$). 
\end{itemize}
In practice, the first ingredient holds uniformly in our examples, while the reason for the second ingredient is more case specific. For concrete examples of this paradigm, see \cref{cor:simplequillenliftingk1x} and \cref{prop:autoql}. In these cases, one also says that $\fp$ originates in $\cO(\fp) \in \Spec^h(\pi_*S_i)$, justifying our choice of notation. 

\begin{proof}[Proof of \Cref{ssec:criterionsql}]
Throughout this proof, we will denote the components of the map $f$ by $f_i\colon R \to S_i$ and repeatedly use the decomposition $\Spec^h(\pi_*\prod_{i \in I}S_i) \cong \coprod_{i\in I}\Spec^h(\pi_*S_i)$ to identify both sides. Moreover, as observed in \cite[Section 5.2]{bchv1}, the support $\supp_S$ is compatible with this decomposition.

Given a module $M \in \Mod_{R}$, we need to show that there is an equality
\begin{equation}\
\res_{f}^{-1}\res_{f}\supp_{S}(f^*M) = \supp_{S}(f^*M),
\end{equation}
leaving the modifications for the dual claim about the cosupport of the coinduction of $M$ to the reader. Since the inclusion ``$\supseteq$'' is true automatically, it remains to prove the reverse inclusion. If $\res_{f}^{-1}\res_{f}\supp_{S}(f^*M)$ is empty, the claim holds vacuously, so we may assume that there is a prime ideal $\fq \in \res_{f}^{-1}\res_{f}\supp_{S}(f^*M)$. Unwinding this statement, there exist $j,k \in I$ with $\fq \in \Spec^h(\pi_*S_j)$ and $\fr \in \Spec^h(\pi_*S_k)$ with $\fr \in \supp_{S_k}(f_k^*M)$ and $\res_{f_j}(\fq) = \res_{f_k}(\fr)$. Our claim is that $\fq \in \supp_{S}(f^*M)$ as well. 

To this end, set $\fp = \res_{f_j}(\fq)$ and consider $\cO(\fp) \in \Spec^h(\pi_*S_i)$. By assumption, there exist morphisms of commutative ring spectra $\zeta(\fq)\colon S_j \to S_i$ and $\zeta(\fr)\colon S_k \to S_i$ under $R$ such that $\res_{\zeta(\fq)}(\cO(\fp)) = \fq$ and $\res_{\zeta(\fr)}(\cO(\fp)) = \fr$, respectively. We may illustrate the situation as follows:
\[
\xymatrix{& S_i \ar@{<--}[ld]_{\zeta(\fq)} \ar@{<-}[dd]_{f_i} \ar@{<--}[rd]^{\zeta(\fr)} & & & & & \cO(\fp) \ar@{~>}[ld]_{\res_{\zeta(\fq)}} \ar@{~>}[dd]_{\res_{f_i}} \ar@{~>}[rd]^{\res_{\zeta(\fr)}} \\
S_j \ar@{<-}[rd]_{f_j} & & S_k \ar@{<-}[ld]^{f_k} & & & \fq \ar@{~>}[rd]_{\res_{f_j}} & & \fr \ar@{~>}[ld]^{\res_{f_k}} \\
& R & & & & & \fp.}
\]
Since $\zeta(\fr)$ is biconservative and both $\Mod_{S_i}$ and $\Mod_{S_k}$ are stratified, the Avrunin--Scott identities are satisfied by \cref{cor:as}, which in turn implies 
\[
\cO(\fp) \in \res_{\zeta(\fr)}^{-1}(\fr) \subseteq \res_{\zeta(\fr)}^{-1}\supp_{S_k}(f_k^*M) = \supp_{S_i}(f_i^*M)),
\]
since $\zeta(\fr)^* \circ f_k^* \simeq f_i^*$. Applying $\res_{\zeta(\fq)}$ then yields 
\[
\fq = \res_{\zeta(\fq)}(\cO(\fp)) \in \res_{\zeta(\fq)}\supp_{S_i}(f_i^*M)) = \supp_{S_j}(f_j^*M) \subseteq \supp_S(f^*M)
\]
by \cref{thm:suppbasechange}(1), using the fact that $f_i^* \simeq \zeta(\fq)^* \circ f_j^*$.
\end{proof}

\subsection{Retractive descent}

The goal of this subsection is to exhibit sufficient conditions that guarantee descent of stratification along a morphism of ring spectra. In particular, we will show that the stratification property is closed under retracts. 

\begin{prop}
Let $f\colon R \to S$ be as above and suppose there exists a map of commutative Noetherian ring spectra $g\colon S \to T$ satisfying the following two conditions:
	\begin{enumerate}
		\item $gf$ is biconservative.
		\item $gf$ satisfies Quillen lifting.
	\end{enumerate}
If $\Mod_S$ is stratified, then $\Mod_R$ is stratified. 
\end{prop}
\begin{proof}
By \Cref{lemma:chouinard-transitivity}, $f$ is biconservative. In light of \Cref{thm:stratification}, it thus remains to show that $f$ satisfies Quillen lifting. To this end, let $M, N \in \Mod_R$ be two modules such that there exists $\fp \in \res_f(\supp_S(f^*M)) \cap \res_f(\cosupp_S(f^!N))$. Since $gf$ is biconservative, \Cref{thm:suppbasechange} shows that
\[
\res_f\supp_S(f^*M) \subseteq \supp_R(M) = \res_{gf}\supp_T((gf)^*M) 
\]
and 
\[
\res_f\cosupp_S(f^!N) \subseteq \cosupp_R(N) = \res_{gf}\cosupp_T((gf)^!N), 
\]
so $\fp$ lies in image of $\res_{gf}$. Let $\cQ_{gf}(\fp)$ be a Quillen lift of $\fp$ along $gf$ in $\Spec^h(\pi_*T)$. We claim that 
\[
\cQ_f(\fp) = \res_g(\cQ_{gf}(\fp))
\]
is then a Quillen lift for $\fp$ along $f$. It is clear that $\res_g(\cQ_{gf}(\fp)) \in \res_f^{-1}(\fp)$. To prove that $\res_g(\cQ_{gf}(\fp)) \in \supp_S(f^*M) \cap \cosupp_S(f^!M)
$, we will first show that
\[
\res_g\supp_T((gf)^*M) \subseteq \supp_S(f^*M) \quad \text{and} \quad \res_g\cosupp_T((gf)^!N) \subseteq \cosupp_S(f^!N).
\]
Indeed, by Proposition 3.14 in \cite{bchv1} applied to $g$ and using that $\Mod_S$ is stratified, there are equalities
\[
\supp_T((gf)^*M) = \supp_T(g^*(f^*M)) = \res_g^{-1}\supp_S(f^*M),
\]
so we get
\[
\res_g\supp_T((gf)^*M) = \res_g\res_g^{-1}\supp_S(f^*M) \subseteq \supp_S(f^*M).
\]
The inclusion for cosupport and coinduction is proven similarly. It follows that 
\[
\res_g(\cQ_{gf}(\fp)) \in \res_g\supp_T((gf)^*M) \cap \res_g\cosupp_T((gf)^!M) \subseteq \supp_S(f^*M) \cap \cosupp_S(f^!N),
\]
hence $\res_g(\cQ_{gf}(\fp))$ is a Quillen lift for $\fp$ along $f$. 
\end{proof}

\begin{cor}\label{cor:retract}
If $f$ admits an $R$-module retract and $\Mod_S$ is stratified, then $\Mod_R$ is stratified.
\end{cor}
\begin{proof}
In the previous proposition, let $g$ be the section to $f$; the claim follows.
\end{proof}

We also have a partial converse to this result:

\begin{prop}\label{rem:products}
Suppose $R = \prod_{i \in I}R_i$ is a finite product of Noetherian commutative ring spectra, so that $\Mod_R \simeq \prod_{i \in I}\Mod_{R_i}$ as symmetric monoidal categories. Then $\Mod_R$ is stratified if and only if the categories $\Mod_{R_i}$ are stratified for all $i \in I$. 
\end{prop}
\begin{proof}
Since $R = \prod_{i \in I}R_i$, there is a corresponding decomposition $\pi_*R \cong \prod_{i \in I}\pi_*R_i$. Moreover, the prime ideals of $\pi_*R$ are of the form $\prod_{i \in I} \fp_i$ where for some $k$ the ideal $\fp_k$ is a prime ideal of $\pi_*R_k$, and for $i \ne k$ we have $\fp_i = \pi_*R_i$. By the construction of $\Gamma_{\fp}$, given a prime ideal $\fp \in \Spec^h(\pi_*R)$ as above, the localizing subcategory $\Gamma_{\fp}\Mod_R$ is minimal if and only if the localizing subcategory $\Gamma_{\fp_k}\Mod_{R_k}$ is minimal. It follows that if $\Mod_{R_i}$ is stratified for each $i \in I$, then so is $\Mod_R$. The \emph{only if} direction is a consequence of \Cref{cor:retract}.
\end{proof}

\subsection{Examples}\label{ssec:stratexamples}

We finish this section with some examples of ring spectra for which we can check that $\Mod_R$ is stratified using the criteria established above. In these cases, it is often easier to check directly that simple Quillen lifting is satisfied when compared to Quillen lifting. 

\begin{ex}
Consider the complexification morphism $f \colon KO \to KU$ from real $K$-theory to complex $K$-theory. On homotopy the induced map 
\[
\xymatrix{\pi_*KO \cong \Z[\eta,\alpha,\beta^{\pm 1}] \ar[r] & \Z[u^{\pm 1}] \cong \pi_*KU}
\]
is given by $\eta \mapsto 0, \alpha \mapsto 2u^2$ and $\beta \mapsto u^4$. It follows that this map is an $\mathcal{N}$-isomorphism, so simple Quillen lifting is satisfied by \Cref{lem:nilpotentCoker}.

The cofiber sequence of $KO$-modules
\[
\Sigma^1 KO \xr{\eta} KO \xr{f} KU
\]
shows that $f$ is a finite morphism, and nilpotence of $\eta$ shows that induction along $f$ is conservative, see \cite[Proposition 5.3.1]{rognes_galois}. By \Cref{rem} coinduction along $f$ is also conservative. Moreover, the homotopy groups of $KU$ are concentrated in even degrees, and satisfy the sufficient conditions for stratification given in \cite[Theorem 1.3]{affineweaklyregular}. It follows that $\Mod_{KU}$ is stratified. By the preceding discussion and \Cref{thm:stratification} we deduce that $\Mod_{KO}$ is stratified. 

We note that the same argument works in the $2$-complete setting, so $\Mod_{KO^{\wedge}_2}$ is stratified. 
\end{ex}

\begin{ex}
Suppose $R$ is a connective (graded) Noetherian commutative ring spectrum with the property that $\pi_iR = 0$ for $i \gg 0$. Consider the morphism $f\colon R \to \tau_{\le 0}R \simeq H\pi_0R$. In the language of \cite{mathew_galois} this is a descendable morphism, i.e., $R\in  \Thick_R( H\pi_0R)$, see \cite[Proposition 3.34]{mathew_galois}. By \Cref{lem:chouinardloc}, the morphism $f$ is biconservative. Moreover, the map $\pi_*f$ is surjective, so $\res_f\colon \Spec^h(\pi_*R) \to \Spec^h(\pi_0R)$ is injective, and hence simple Quillen lifting is satisfied. Note that since $\pi_*R$ is assumed to be graded Noetherian, $\pi_0R$ is Noetherian, and each $\pi_iR$ is finitely-generated as a $\pi_0R$-module \cite[Lemma 4.12]{brunsgub}. By Neeman's result \cite[Theorem 2.8]{neemanchromtower} and the equivalence $\Mod_{H\pi_0R} \simeq \cD(H\pi_0R)$, we thus see that $\Mod_{H\pi_0R}$ is stratified. It follows from \Cref{thm:stratification} that $\Mod_R$ is also stratified.
\end{ex}

\section{Cochains on classifying spaces}\label{sec:kmkm}
In this section, we specialize the work on stratification in the previous sections to the case of cochains on a space $X$, in particular, the case where $X$ is the classifying space of a topological group. Using work of Broto and Kitchloo \cite{broto_kitchloo} we prove Theorem A of the introduction, namely that $\Mod_{C^*(BK)}$ is stratified for $K$ a Kac--Moody group. 

\subsection{Cochains on topological spaces}\label{ssec:cochains}
Let $R$ be a ring spectrum $R$ and let $X$ be a topological space. We write $C^*(X;R)$ for the ring spectrum of $R$-valued cochains on $X$, i.e., the function spectrum $\Hom_{\Sp}(\Sigma_+^{\infty}X,R)$. In particular, there is an isomorphism $\pi_*C^*(X;R) \cong R^{-*}(X)$, the $R$-cohomology of $X$. The multiplication of two elements $x,y \in C^*(X;R)$ is given by the composite 
\[
\xymatrix{\Sigma_+^{\infty}X \ar[r]^-{\Delta} & \Sigma_+^{\infty}X \otimes \Sigma_+^{\infty}X \ar[r]^-{x \otimes y} & R \otimes R \ar[r]^-m & R,} 
\]
where the first map is induced by the diagonal of $X$ and the last map $m$ is the multiplication on the ring spectrum $R$. If $R$ is commutative, then $C^*(X;R)$ inherits the structure of an augmented commutative $R$-algebra.

In this paper, we are mostly interested in the case when $R =Hk$ is the Eilenberg--Mac Lane spectrum of a field $k$ of characteristic $p >0$ for a fixed prime $p$ and that $X$ is a connected space with $H^*(X;k)$ Noetherian. If $X$ is additionally assumed to be $p$-good, then the canonical map $X \to X^\wedge_p$ induces an equivalence $C^*(X^\wedge_p;k)  \simeq C^*(X;k)$ of commutative ring spectra, see for example \cite[Lemma 4.8]{bchv1}. When $k=\F_p$, we will usually omit it from the notation.

\subsection{Chouinard's theorem for classifying spaces}
We recall that in \Cref{sec:quillen} we isolated two conditions to descend stratification along a morphism $f \colon R \to S$ of ring spectra, namely that induction and coinduction along $f$ are conservative, and that $f$ satisfies simple Quillen lifting. Suppose now that $R = C^*(BG)$ for a topological group $G$. In order to construct a suitable morphism $f \colon C^*(BG) \to S$ as above, we  associate to $G$ a category, first considered by Quillen \cite{quillen_stratification}. 

\begin{defn}\label{defn:quillencat}
Let $G$ be a topological group. The category $\cal{A}_p(G)$ has objects the elementary abelian $p$-subgroups of $G$, and morphisms are group homomorphisms induced by subconjugation in $G$. 
\end{defn}

Suppose that $\cal{A}_p(G)$ is equivalent to a finite category. This is true, for example, when $G$ is a compact Lie group \cite[Lemma 6.3]{quillen_stratification}. Let $\cal{E}_p(G)$ denote a set of representatives of conjugacy classes of elementary abelian $p$-subgroups of $G$, then the inclusions $E \le G$ define a morphism
\[
\xymatrix{q_G \colon C^*(BG) \ar[r] & \prod_{E \in \cal{E}_p(G)} C^*(BE).}
\]
We would like to know when induction and coinduction along $q_G$ are conservative. 
\begin{rem}
 When $G$ is a finite group, this is essentially Chouinard's theorem~\cite{chouinard} about projective representations of groups. Indeed, let $\Mod_{kG}$ be the abelian category of $k$-linear $G$-representations. Chouinard proves that $M \in \Mod_{kG}$ is projective if and only if the restrictions of $M$ to all elementary abelian subgroups of $G$ are projective. Equivalently, an object $M$ in the stable module category $\StMod_{kG}$ of $G$ is trivial if and only if it becomes trivial in $\StMod_{kE}$ for $E \in \cE_p(G)$. In this form, Chouinard's theorem can be reinterpreted homotopically, at least for finite $p$-groups: $\StMod_{kG}$ embeds fully faithfully into the homotopy category $K(\Inj_{kG})$ of unbounded injective $kG$-modules, which in turn is equivalent to $\Mod_{C^*(BG)}$:
\[
\xymatrix{\StMod_{kG} \ar@{^{(}->}[r] & K(\Inj_{kG}) \ar[r]^-{\sim} & \Mod_{C^*(BG)}.}
\]
Benson, Iyengar, and Krause~\cite[Proposition 9.6]{bik_finitegroups} then prove that Chouinard's theorem extends to the statement that the functor 
\[
\xymatrix{q_G^*\colon \Mod_{C^*(BG)} \ar[r] & \prod_{E \in \cE_p(G)}\Mod_{C^*(BE)}}
\]
is conservative, where $q_G^*$ denotes induction along the map $q_G \colon C^*(BG) \to \prod_{E \in \cal{E}_p(G)} C^*(BE)$ induced by the subgroup inclusions $E \to G$.  
 \end{rem} 

\begin{ex}\label{ex:ChouinardcompactLie}
When $G$ is a compact Lie group, Benson and Greenlees show in \cite[Theorem 3.1(i)]{bg_stratifyingcompactlie} that $C^*(BG)$ is in the thick subcategory generated by the $C^*(BG)$-modules $C^*(BE)$ for $E \le G$, hence induction and coinduction along $q_G$ are conservative by \Cref{lem:chouinardloc}. 
\end{ex}

\begin{rem}
\Cref{ex:ChouinardcompactLie} shows that conservativity of induction and coinduction along a morphism $f \colon R \to S$ can be satisfied even when $S$ is a not a compact $R$-module. Indeed, if $\pi_0G$ is not a finite $p$-group and $E \le G$ is some elementary abelian, then $C^*(BE)$ is not necessarily finite as a $C^*(BG)$-module. An explicit example is given by $\Z/2 \times \Z/2 \leq A_4$. 
\end{rem}

\begin{ex}\label{ex:isotropy}
Let $G$ be a compact Lie group, $X$ a finite $G$-CW complex. As a generalization of the previous example, consider the category $\cal{A}_G(X)$ whose objects are homotopy classes of $G$-equivariant maps $\iota\colon G/E \to X$, where $E$ is an elementary abelian $p$-subgroup of $G$, and morphisms from $\iota \colon G/E \to X$ to $\iota' \colon G/E' \to X$ are $G$-equivariant maps $f\colon G/E\rightarrow G/E'$ such that $\iota'\circ f\simeq \iota$ (see \cite[Section 8]{quillen_stratification}). For any $(\iota\colon G/E \to X) \in \cal{A}_G(X)$, there is an induced map  $C^*(X_{hG}) \to C^*(BE)$ obtained by applying cochains to the Borel construction on $\iota$. These maps assemble together to give a map
\[
\xymatrix{q_{G,X} \colon C^*(X_{hG}) \ar[r] & \prod_{E \in \cal{E}_G(X)} C^*(BE),}
\]
where $\cal{E}_G(X)$ denotes a set of representatives of isomorphism classes of objects in $\cal{A}_G(X)$. Using the techniques of Benson and Greenlees \cite[Theorem 3.1(i)]{bg_stratifyingcompactlie}, Cameron \cite[Proposition 5.4.8]{cameron_2018} has shown that $q_{G,X}$ is biconservative. In particular, \cite[Lemma 5.4.11]{cameron_2018} shows that $C^*(X_{hG})$ is in the thick subcategory of $\Mod_{C^*(X_{hG})}$ generated by the set of $C^*(X_{hG})$-modules $\{ C^*(EG \times_G G/E) \colon (G/E \to X) \in \cal{A}_G(X)\}$, so that \Cref{lem:chouinardloc} applies. Note that if $X$ is contractible, then one recovers \Cref{ex:ChouinardcompactLie}. One can also show that $q_{G,X}$ satisfies Quillen lifting (this is essentially \cite[Proposition 5.4.13]{cameron_2018}), and so deduce that $\Mod_{C^*(X_{hG})}$ is stratified by $H_G^*(X)$, see \cite[Theorem 5.4.1]{cameron_2018}.

Taking $X = G$, considered as a $G$-space with conjugation action, the Borel construction $G_{hG}$ is homotopy equivalent to $\Lambda (BG)$, the free loop space on $BG$ \cite{Smith1981characteristic}. 
\end{ex}

\begin{ex}\label{ex:fusion}
Suppose $\cal{F}$ is a saturated fusion system on a finite $p$-group $S$ (see \cite[Section 1]{blo_fusion}).   Broto, Levi and Oliver introduced the notion of a classifying space for a fusion system, and Chermak \cite{chermak_existence} showed that it always exists and is unique up to homotopy equivalence. Then we can associate a classifying space $B\cal{F}$ which comes equipped with a canonical map $\theta \colon BS \to B\cal{F}$, inducing  $\theta^{\ast} \colon C^*(B\cal{F}) \to C^*(BS)$ which is split as a map of $C^*(B\cal{F})$-modules \cite{ragnarsson_transfer,rag_transfer_diagram}. 

Let $\cal{E}_p(\cal{F})$ (resp. $\cal{E}_p(S)$) be a set of representatives of $\cal{F}$-conjugacy (resp. $S$-conjugacy) classes of elementary abelian $p$-subgroups in $S$.  
We obtain a morphism  
\[
\xymatrix{q_{\cal{F}} \colon C^*(B\cal{F}) \ar[r] & \prod_{E \in \cal{E}_p(\cal{F})} C^*(BE)}
\]
induced by the inclusions $\iota_E\colon E\leq S$. Those representatives can be chosen in a way that the map $\pi\colon \cE_p(S)\to \cE_p(\cF)$, which assigns a representative in the $\cal{F}$-conjugacy class, is a surjection. Consider the resulting decomposition $\cE_p(S) = \coprod_{E \in \cE_p(\cF)}\pi^{-1}(E)$. For each $V\in \pi^{-1}(E)$ there is an isomorphism $f\colon V\rightarrow E$ in $\cal{F}$. We can assemble them in a morphism 
\[
\xymatrix{q_E \colon C^*(BE) \ar[r] & \prod_{V \in \pi^{-1}(E)} C^*(BV).}
\]
Since $f\in \cal{F}$, we have $f^*\circ \iota_E^*\circ \theta^*\simeq \iota_V^*\circ \theta^*$, so it follows that the following diagram commutes: 
\[
\xymatrix{C^*(B\cal{F}) \ar[r]^-{\theta^\ast} \ar[rd]_{q_\cal{F}} & C^*(BS)  \ar[r]^-{q_S} & \prod_{V \in \cE_p({S})} C^*(BV) \ar[d]^-{\simeq} \\
&\prod_{E \in \cal{E}_p(\cal{F})} C^*(BE) \ar[r]^-{\prod q_E} & \prod_{E \in \cal{E}_p(\cal{F})}\prod_{V \in \pi^{-1}(E)} C^*(BV).}
\]
Since $\theta^{\ast}$ is split, induction and coinduction along it are conservative by \Cref{lem:chouinardloc}. Moreover, $q_S$ is biconservative, because $S$ is a finite $p$-group (see \Cref{ex:ChouinardcompactLie}). By transitivity, the composite is therefore biconservative as well. But by \Cref{lemma:chouinard-transitivity} this implies that $q_\cal{F}$ is also biconservative.
\end{ex}

\begin{rem}
	More generally, one can consider saturated fusion systems $\cal{F}$ on a discrete $p$-toral group $S$, i.e., a group that fits in an extension
	\[
\xymatrix{1 \ar[r] & (\Z/p^{\infty})^r \ar[r] & S \ar[r] & \pi \ar[r] & 1}
\]
where $r \ge 0$ is finite, and $\pi$ is a finite $p$-group. As in the case $r = 0$ considered in \Cref{ex:fusion}, there is an associated classifying space $B\cal{F}$ \cite{blo_pcompact,ll_uniqueness} and a canonical map $\theta \colon BS \to B\cal{F}$. The induced map $\theta^* \colon C^*(B\cal{F}) \to C^*(BS)$ is split as a map of $C^*(B\cal{F})$-modules by \cite[Proposition 4.24]{bchv1}. The composite 
\[
\xymatrix{C^*(B\cal{F}) \ar[r]^-{\theta^\ast} & C^*(BS) \ar[r]^-{q_S} & \prod_{V \in \cal{E}_p(S)}C^*(BV)}
\]
is known to be biconservative when $\cal{F}$ models a compact Lie group, a connected $p$-compact group, or when $S$ is a finite $p$-group as in \Cref{ex:fusion}, see \cite[Corollary 4.20 and Theorem 4.25]{bchv1}. We also have the map $q_{\cal{F}}$ as in \Cref{ex:fusion}, and a similar argument shows that $q_{\cal{F}}$ is biconservative if and only if $q_S$ is biconservative.
\end{rem}

\subsection{A hierarchy of spaces}\label{ssec:hierarchy}

In \cite{kropholler_mislin} Kropholler and Mislin introduced a class of hierarchically decomposable groups. Given a class $\mathcal X$ of groups, 
they define $H_1\mathcal X$ to be the class of groups $G$ which admit a finite dimensional $G$-CW-complex $X$ with cell stabilizers in $\mathcal X$.  Then $H\mathcal X$ is defined as the smallest class $\mathcal J$ containing $\mathcal X$ such that $H_1\mathcal J=\mathcal J$. Of special interest is the class $\mathcal X$ of 
finite groups. 

Inspired by this, Broto and Kitchloo \cite{broto_kitchloo} introduced the following class of groups. Let $\mathcal X$ be a class of compactly generated Hausdorff topological groups and let $p$ be a fixed prime.
Define $\mathcal K_1 \mathcal X$ to the class of compactly generated Hausdorff
topological  groups $G$ for which there exists a finite $G$-CW-complex $X$ such that:
\begin{enumerate}
\item the isotropy subgroups of $X$ belongs to $\mathcal X$, and
\item for each finite $p$-subgroup $\pi <G$, the fixed point space $X^\pi$ is $p$-acyclic.
\end{enumerate}
 
 \begin{rem}\label{rem:Iso}
 The second condition in the definition applied to the trivial subgroup implies that  $X$ is mod $p$ acyclic and then $C^*(X_{hG})\simeq C^*(BG)$. Moreover, from the proof of \cite[Lemma 5.4.10]{cameron_2018} we see that  $C^*(BG)\in \Thick(C^*(BH)|H\in \operatorname{Iso}_G(X))$, where $\operatorname{Iso}_G(X)$ denotes a set of representatives of $G$-conjugacy classes of isotropy subgroups of the $G$-action on $X$ (which are compactly generated Hausdorff topological groups by the definition of $\cal{K}_1\cal{X}$).
 \end{rem}
  
In what follows $\mathcal X$ is the class of compact Lie groups. Examples of groups in $\mathcal K_1 \mathcal X$ are given by Kac--Moody groups and infinite Coxeter groups, see \Cref{ssec:kacmoody}. For any $G \in \mathcal K_1 \mathcal X$ we let $\cal{A}_p(G)$ denote the Quillen category associated to $G$, and $\cal{E}_p(G)$ a set of representatives for conjugacy classes of elementary abelian $p$-subgroups. As with the case of a compact Lie group, we can construct a morphism 
\[
\xymatrix{q_G \colon C^*(BG) \ar[r] & \prod_{E \in \cal{E}_p(G)}C^*(BE).} 
\] 

\begin{prop}\label{prop:indcoinductionk1x}
Let $G\in \mathcal K_1 \mathcal X$, then induction and coinduction along $q_G$ are conservative. 
\end{prop}

\begin{proof}\sloppy
Since $G \in \cal K_1 \mathcal X$, by definition, there exists a finite $G$-CW-complex $X$ whose isotropy subgroups are compact Lie groups, and such that $C^*(X_{hG}) \simeq C^*(BG)$.  Let $E\leq G$ be an elementary abelian $p$-subgroup, then $X^E$ is mod $p$ acyclic by hypothesis, and in particular is non-empty. Then $E\leq G$ is a subgroup of an isotropy subgroup of $G$, hence $E$ is subconjugated to some $H\in \operatorname{Iso}_G(X)$, where the latter is a set of representatives of $G$-conjugacy classes of isotropy subgroups of $X$. Let $\cE_p(H)$ be a set of representatives of $H$-conjugacy classes of elementary abelian $p$-subgroups in $H$ and write $\cal{E}_{\mathcal X}(G)=\{(H,E)|H\in  \operatorname{Iso}_G(X),\; E\in \mathcal E_p(H)\}$. Therefore, one can choose representatives in a way so that the natural map $\cal{E}_{\mathcal X}(G) \rightarrow \cal{E}_p(G)$ is surjective.

The composite  
\[
\xymatrix{q\colon C^*(BG) \ar[r]^-{q_{\operatorname{Iso}}} & \prod_{H \in \operatorname{Iso}_G(X)}C^*(BH) \ar[r]^-{\prod{q_H}} &  \prod_{H \in \operatorname{Iso}_G(X)}\prod_{\mathcal E_p(H)} C^*(BE),}
\]  
is biconservative: $q_{\operatorname{Iso}}$ is so by \Cref{rem:Iso} and $q_H$ by \Cref{ex:ChouinardcompactLie} since $H\leq G$ are compact Lie groups. Now the map $q$ factors through $q_G\colon C^*(BG) \to   \prod_{\mathcal E_p(G)} C^*(BE) $ by the same argument as in \Cref{ex:fusion} using conjugation by elements in $G$ instead of morphisms in the fusion systems. We conclude that $q_G$ is also biconservative.
\end{proof}

Let $\cV_{G}$ denote the homogeneous prime ideal spectrum of $H^*(BG)$. For an elementary abelian subgroup $E \leqslant G$, let $\cV_{E}^+ = \cV_{E} \setminus \bigcup_{E' < E} \res_{E}^{E'}\cV_{E'}$, where $\res_{E}^{E'}\colon \cV_{E'} \to \cV_{E}$ is the restriction map. Finally, we write $\cV_{G,E}^{+}= \res_{G}^{E} \cV_{E}^+$, to denote the restriction of $\cV_{E}^+$ induced by the map $H^*(BG) \to H^*(BE)$. 

\begin{thm}\label{thm:kmhfiso}
 Let $\mathcal X$ be the class of compact Lie groups and $G\in \mathcal K_1 \mathcal X$.
\begin{enumerate}
 \item (Broto--Kitchloo) $H^*(BG)$ is Noetherian and the Quillen map 
\[
\xymatrix{q_G\colon H^*(BG) \ar[r] & \displaystyle\varprojlim_{E\in \cal{A}_p(G)} H^*(BE)}
\]
is an $F$-isomorphism.
\item (Rector) The variety $\mathcal V_G$ admits a decomposition 
\[
\mathcal V_G = \coprod_{E\in \mathcal E_p(G)}\mathcal V^+_{G,E}
\]
where $\mathcal E_p(G)$ denotes a set of representatives for conjugacy classes of elementary abelian $p$-subgroups of $G$. 
\end{enumerate}
\end{thm}
\begin{proof}
The first part is proved by Broto and Kitchloo. The fact that  $H^*(BG)$ is Noetherian is shown in \cite[Theorem 4.8]{broto_kitchloo}, while the $F$-isomorphism is proved in \cite[Theorem 4.1]{broto_kitchloo}. 

For the second part, we can use the methods of Rector \cite{rector_quillenstrat} to deduce the decomposition of the variety $\mathcal V_G$ from the ${F}$-isomorphism theorem as in \cite[Theorem 5.6]{bchv1}.  In particular, we claim that the Conditions (1) -- (5) of \cite[Proposition 2.3]{rector_quillenstrat} are satisfied for the pair $(\cal{A}_p(G),H^*(B-))$, with proofs analogous to that given in the proof of Theorem 5.6 of \cite{bchv1}. 

It follows from \cite[Proposition 2.5 and Theorem 2.6]{rector_quillenstrat} that $\Lambda = \varprojlim_{E\in \cal{A}_p(G)} H^*(BE)$ is a reduced, Noetherian unstable algebra over the Steenrod algebra. The proof of \cite[Theorem 2.6]{rector_quillenstrat} shows moreover that $\cal{V}_{\Lambda} \cong \coprod_{E\in \mathcal E_p(G)}\mathcal V^+_{G,E}$. Since an $F$-isomorphism induces an isomorphism on varieties, we deduce from this and the first part of the theorem that 
\[
\cal{V}_G \cong \cal{V}_{\Lambda} \cong  \coprod_{E\in \mathcal E_p(G)}\mathcal V^+_{G,E},
\]
  as required.
\end{proof}

We also need the next result---the stratification result follows from \cite{bik_finitegroups}, as in \cite[Theorem 4.2(i)]{bg_stratifyingcompactlie}, while the Avrunin--Scott identities are then a consequence of \Cref{cor:as}. 

\begin{thm}[Benson--Iyengar--Krause]\label{thm:stratcompactLie}
 For any elementary abelian $p$-subgroup $E$, the category $\Mod_{C^*(BE)}$ is stratified by $H^*(BE)$.
Moreover,  for any morphism $E' \to E$ of elementary abelian $p$-subgroups with induced morphism $f \colon C^*(BE) \to C^*(BE')$, the Avrunin--Scott identities are satisfied, i.e., 
\[
\begin{split}
\supp_{C^*(BE')}(f^* M) &= \res_{f}^{-1}\supp_{C^*(BE)}(M) \quad \\
\cosupp_{C^*(BE')}(f^! M) &= \res_{f}^{-1}\cosupp_{C^*(BE)}(M)
\end{split}
\]
for all $M \in \Mod_{C^*(BE)}$. 
\end{thm}

\begin{cor}\label{cor:simplequillenliftingk1x}
Let $G\in \mathcal K_1 \mathcal X$, then 
\[
\xymatrix{q \colon C^*(BG) \ar[r] & \displaystyle\prod_{E\in \mathcal  E_p(G)} C^*(BE)}
\]
satisfies simple Quillen lifting. 
\end{cor}
\begin{proof}
The argument is similar to the one given in \cite[Theorem 5.10]{bchv1}, for which we refer the reader for more details. We have to verify the conditions of \cref{ssec:criterionsql} applied to the morphism $q$. By \cref{thm:stratcompactLie}, $\Mod_{C^*(BE)}$ is stratified for all elementary abelian groups $E$. Quillen stratification as in \cref{thm:kmhfiso} provides a surjection
\[
\coprod_{E\in \cE_p(G)}\cV_{E} \twoheadrightarrow \coprod_{E\in \mathcal E_p(G)}\cV^+_{G,E} = \cV_G.
\]
Any set-theoretic section to this map then gives a choice of function $\cO\colon \cV_G \to \coprod_{E\in \cE_p(G)}\cV_{E}$ which satisfies the condition ($\clubsuit$). Indeed, if $\fp \in \cV_G$, the strong form of Quillen stratification of \cref{thm:kmhfiso}(2) implies that there exists a unique elementary abelian subgroup $E \subseteq G$ in which $\fp$ originates and that, furthermore, the prime ideal $\cO(\fp) \in \cV_E^+ \subseteq \cV_E$ is unique up to the action of the Weyl group. Both the inclusions among elementary abelian subgroups and the Weyl group action give morphisms of commutative ring spectra between the corresponding cochain algebras under $C^*(BG)$, so the claim follows.
\end{proof}

Combining the previous results we obtain the following. 
\begin{thm}\label{thm:stratk1x}
  Let $\mathcal X$ denote the class of compact Lie groups. Then, for any $G \in \mathcal K_1 \mathcal X$, the category $\Mod_{C^*(BG)}$ is stratified by $H^*(BG)$. 
\end{thm}
\begin{proof}
We apply \Cref{thm:stratification} to the morphism   
\[
\xymatrix{q_G \colon C^*(BG) \ar[r] & \prod_{\cE_p(G)} C^*(BE).}
\] 
Note that $\cE_p(G)$ is finite, see Theorem 2.1 of \cite{broto_kitchloo} and the discussion at the bottom of page 630. By \Cref{cor:simplequillenliftingk1x,prop:indcoinductionk1x} stratification of $\Mod_{C^*(BG)}$ will follow if $\prod_{\mathcal E_p(G)} \Mod_{C^*(BE)}$ is stratified, but this follows from \Cref{rem:products,thm:stratcompactLie}.
\end{proof}

By \Cref{thm:tel_conj}, we deduce that the telescope conjecture holds for $\Mod_{C^*(BG)}$ with $G \in \mathcal K_1 \mathcal X$.

\begin{ex}[Stratification for Kac--Moody groups]\label{ssec:kacmoody}
Unitary forms of Kac--Moody groups can be understood as an extension of the class of compact Lie groups with whom they share many properties. For example, they have a maximal torus of finite rank and a notion of Weyl group which is generated by reflections, see \cite{ku}. The Weyl group of a Kac--Moody group is a Coxeter group which is, in general, infinite. 

The study of the topology of Kac--Moody groups and their classifying spaces was initiated by Kitchloo in his thesis and followed by many authors, see for example \cite{nitu_thesis,broto_kitchloo,MR1994344,kitchloo_kacmoody}. In fact, the interest in the groups $G \in \cal{K}_1\cX$ studied in the previous section arose because Kac--Moody groups give examples of groups in $\cal{K}_1\cX$, see \cite[Section 5]{broto_kitchloo}.
\begin{thm}[Broto--Kitchloo]
  Let $\mathcal X$ denote the class of compact Lie groups, then any Kac--Moody group $K$ is in $\cK_1 \cX$. Moreover, the Weyl group of $K$ belongs to the class $\cK_1 \cX$. 
\end{thm}
When combined with \Cref{thm:stratk1x}, we deduce the following. 
\begin{thm}\label{thm:kacmoodystrat}
  Let $K$ be a Kac--Moody group and $W$ the Weyl group of $K$. Then $\Mod_{C^*(BK)}$ and $\Mod_{C^*(BW)}$ are stratified.
\end{thm}
\end{ex}

\section{Hopf spaces with Noetherian mod $p$ cohomology}\label{sec:noetherian}

The structure of $H$-spaces with Noetherian mod $p$ cohomology is quite well understood. Using Lannes' $T$-functor and Bousfield nullification functors, those spaces fit into a fibration where the base and the fibre are finite and Eilenberg--MacLane $H$-spaces, respectively, see \cite{bcs_deconstructing} and \cite{ccs_deconstructing}.  In this section, we use this decomposition and the work in the previous sections to show that for a connected $H$-space $X$ with Noetherian mod $p$ cohomology ring, the category of modules over $C^*(X)$ is stratified.

\subsection{Unstable algebras over the Steenrod algebra}\label{ssec:lannes}

The mod $p$ cohomology of any space $X$ is equipped with an action of the mod $p$ Steenrod algebra $\mathcal{A}_p$. In fact, it is a graded $\mathcal{A}_p$-module, and the action of $\mathcal{A}_p$ satisfies additional properties, such that $H^*(X;\F_p)$ is an object in the abelian category of unstable modules over the Steenrod algebra $\mathcal{U}$, see \cite{Schwartz_book} for example. We let 
$\mathcal{K}$ denote the category of unstable algebras over the Steenrod algebra; that is, an object $K \in \mathcal{K}$ is an unstable module together with maps $\rho \colon K \otimes K \to K,\, \eta \colon \F_p \to K$ satisfying some additional axioms, again see \cite{Schwartz_book}. The diagonal of $X$ equips $H^*(X;\F_p)$ with a multiplication, which satisfies certain compatibility relations with the Steenrod algebra, so that the mod $p$ cohomology of any space is in fact an object of $\cal{K}$.  

We recall that a morphism $f \colon R \to S$ of $
\F_p$-algebras is said to be an $F$-monomorphism if every element in $\ker(f)$ is nilpotent, and an $F$-epimorphism if for every $s \in S$, there exists a natural number $n$ such that $s^{p^n} \in \im(f)$. We say that $f$ is an $F$-isomorphism if it is both an $F$-monomorphism and an $F$-epimorphism. If the number $n$ can be chosen independently of $s$, then we say that it is a uniform $F$-isomorphism. Lannes' machinery gives a criterion for when a morphism of unstable algebras is an $F$-monomorphism or $F$-epimorphism. 

\begin{thm}\cite[Corollary II.1.4]{hls_categoryU}\label{thm:fepimono}
  Let $\phi \colon K \to K'$ be a morphism of unstable algebras. Then the morphism of $\F_p$-algebras underlying $\phi$ is an $F$-monomorphism (resp. $F$-epimorphism) if and only if 
\[
\xymatrix{\Hom_{\cal{K}}(K',H^*(BV)) \ar[r] & \Hom_{\cal{K}}(K,H^*(BV))}
\]
is surjective (resp. injective) for all elementary abelian $p$-groups $V$. 
\end{thm}

Given a topological space $X$, taking cohomology induces a morphism
\[
\xymatrix{[BV,X] \ar[r] &  \Hom_{\cal{K}}(H^*(X),H^*(BV))}
\] 
and Lannes' theory \cite{La}  establishes conditions under which it is a bijection. Recall that we write $X^\wedge_p$ for the Bousfield--Kan $p$-completion of a space $X$. 

\begin{thm}[Lannes]\label{thm:lannespi0comp}
Let $X$ be a topological space such that $H^*(X;\F_p)$ is of finite type, 
then taking mod $p$ cohomology induces a bijection
\[
\xymatrix{[BV,X^\wedge_p] \ar[r]^-{\simeq} & \Hom_{\mathcal K}(H^*(X), H^*(BV)).}
\]
Moreover, if $\pi_1X$ is a finite $p$-group, then $[BV,X]\cong [BV,X^\wedge_p]$. \end{thm}

\begin{proof}
The first part of the statement is \cite[Theorem 3.1.1]{La}. For the second, by the obstruction theoretic argument of \cite[proof of Theorem 3.1]{dwyerzabrodsky_classifyingspaces}, the natural map $X\to X^\wedge_p$ induces an isomorphim $[BV,X]\cong[BV,X^\wedge_p]$. 
\end{proof}

One can combine \Cref{thm:fepimono,thm:lannespi0comp} into the following.

\begin{prop}\label{prop:simplequillenspaces}
Let $X$ and $Y$  be  topological spaces such that $H^*(X)$ and  $H^*(Y)$ are of finite type.  Let $f \colon X \to Y$ be a map such that \[
\xymatrix{[BV,X^{\wedge}_p] \ar[r] & [BV,Y^{\wedge}_p]}
\]
is injective (resp. surjective) for all elementary abelian $p$-groups $V$, then $f^*\colon H^*(Y)\rightarrow H^*(X)$ is an $F$-epimorphism (resp. $F$-monomorphism).
\end{prop}

Finally, we state a theorem proven by Miller~\cite[Theorem 1.5]{miller_sullivan} that allows us to remove $p$-completions when mapping from classifying spaces of elementary abelian $p$-groups.

\begin{thm}[Miller]\label{thm:miller}
Let $X$ be a connected nilpotent space, then the canonical map $X \to X^\wedge_p$ induces an isomorphism
\[
\xymatrix{\map_*(BV, X) \ar[r]^-{\simeq} & \map_*(BV, X^\wedge_p).}
\]
\end{thm}

\begin{cor}
Let $X$ be a connected nilpotent space with finite fundamental group, then the canonical map $X \to X^\wedge_p$ induces an equivalence
\[
\xymatrix{[BV, X] \ar[r]^-{\simeq} & [BV, X^\wedge_p].}
\]
\end{cor}
\begin{proof}
We recall that if $Y$ is a path-connected space, the forgetful map $[BV,Y]_* \to [BV,Y]$ is surjective and two elements have the same image if and only if they are in the same orbit by the action of $\pi_1(Y)$.  Now consider the following commutative diagram
\[ 
\xymatrix{[BV,X]_* \ar[r] \ar[d]_{\alpha}^-{\cong} & [BV,X]  \ar[r] \ar[d] & \pi_0(X) = \ast \ar[d]^-{\cong} \\
[BV,X^\wedge_p]_* \ar[r] & [BV,X^\wedge_p]  \ar[r] & \pi_0(X^\wedge_p) = \ast,}
\]
in which the map $\alpha$ is $\pi_1(X)$-equivariant. The diagram shows that $[BV,X]\to[BV,X^{\wedge}_p]$ is surjective. To prove injectivity, one needs to prove that given two pointed maps $f,g \colon BV \to X$ such that $f^{\wedge}_p,g^{\wedge}_p\in [BV,X^{\wedge}_p]$ differ by the action of $\pi_1(X^{\wedge}_p)$, then $f$ and $g$ differ by the action of $\pi_1(X)$. This is the case if $\pi_1(X)\to \pi_1(X^{\wedge}_p)$ is surjective, for example, when $\pi_1(X)$ is finite. 
\end{proof}

\subsection{$B\Z/p$-nullification and nilpotent fibrations}\label{ssec:nullification}

We collect some material about $B\Z/p$-nullification and nilpotent fibrations that we will use later in this section. Our main references for the former are \cite{drorfarjoun_localization, bousfield_unstablelocalization}, while we refer to \cite{bousfield_kan} for the latter. 

Let $A$ be a space. Dror-Farjoun \cite{drorfarjoun_localization} and Bousfield \cite{bousfield_unstablelocalization} construct a localization functor $P_A\colon \Top \to \Top$ that universally inverts the map $A \to \ast$, the so-called $A$-nullification. In more detail, a space $X$ is $A$-null if the canonical map $A \to \ast$ induces a weak equivalence
\[
\xymatrix{X \simeq \map(\ast,X) \ar[r]^-{\sim} & \map(A,X).}
\]
If $A$ and $X$ are connected, this is equivalent to a weak equivalence $\ast \simeq \map_*(A,X)$ on pointed mapping spaces. The functor $P_A$ is coaugmented and takes values in $A$-null spaces. Moreover, if $X$ is a space and $Y$ is an $A$-null space, then the coaugmention of $P_A$ induces a weak equivalence 
\[
\xymatrix{\map(P_AX,Y) \ar[r]^-{\sim} & \map(X,Y),} 
\]
i.e., $P_A$ is left adjoint to the inclusion of the category of $A$-null spaces into $\Top$. It follows that $P_A$ preserves finite products and hence $H$-spaces; in fact, if $X$ is an $H$-space, then $X \to P_AX$ is an $H$-map of $H$-spaces. In \cite[Proposition 2.9]{bousfield_unstablelocalization}, Bousfield proves that $P_A$ preserves the property of being simply connected. 

\begin{rem}
It is an open problem (see \cite[Question 9.F.7]{drorfarjoun_localization}) whether $P_A$ also preserves nilpotency.   
\end{rem}

We now specialize to the case $A = B\Z/p$. Miller shows that $\map_{\ast}(B\Z/p,X)$ is weakly contractible for any nilpotent space $X$ with bounded mod $p$ cohomology, see \cite[Theorem C]{miller_sullivan}, so we get:

\begin{thm}\label{thm:millersullivan}
A connected nilpotent space with bounded mod $p$ cohomology is $B\Z/p$-null. 
\end{thm}
\begin{rem}\label{rem:h_space_null}
	By \Cref{thm:miller}, if $X$ is a connected $H$-space, then there are weak equivalences
\[
\ast \simeq \map_*(B\Z/p,P_{B\Z/p}X) \simeq \map_*(B\Z/p,(P_{B\Z/p}X)^\wedge_p),
\]
so $(P_{B\Z/p}X)^\wedge_p$ is $B\Z/p$-null as well. 
\end{rem}

Recall that a fibration $f\colon E \to B$ of connected spaces is called nilpotent if the fiber of $f$ is connected and the canonical action of $\pi_1E$ on $\pi_iF$ is nilpotent for all $i \ge 1$. In this case, we will also refer to 
\[
\xymatrix{F \ar[r] & E \ar[r] & B}
\]
as a nilpotent fiber sequence. Note that a space $X$ is nilpotent if and only if $X \to \ast$ is a nilpotent fibration. Furthermore, the following two-for-three property holds for nilpotent fibrations: If 
\[
\xymatrix{X_2 \ar[r]^-{f_2} & X_1 \ar[r]^-{f_1} & X_0}
\]
are fibrations with connected fibers, then if two of the three maps $f_1, f_2, f_1f_2$ are nilpotent fibrations, then so is the third, see \cite[Proposition II.4.4]{bousfield_kan}. In particular, if $f\colon E \to B$ is a fibration with $E$ and $B$  both nilpotent, then $f$ is a nilpotent fibration. 

In \cite[Section II.4.8]{bousfield_kan}, it is shown that the $p$-completion of a nilpotent fiber sequence is again a nilpotent fiber sequence, i.e., if $f\colon E \to B$ is a nilpotent fibration, then $f^\wedge_p\colon E^\wedge_p \to B^\wedge_p$ is a fibration with fiber $F^\wedge_p$. In particular, if $X$ is nilpotent, then so is $X^\wedge_p$.

We collect these results in a lemma for later use. 

\begin{lem}\label{lem:fiberseq}
Let $f\colon E \to B$ be a nilpotent fibration of connected spaces with fiber $F$ such that $E$ is nilpotent and $B$ is $B\Z/p$-null and of finite type, then $p$-completion gives a nilpotent fiber sequence
\[
\xymatrix{F^\wedge_p \ar[r] & E^\wedge_p \ar[r] & B^\wedge_p}
\]
with $B\Z/p$-null base $B^\wedge_p$.
\end{lem}
\begin{proof}
It follows from the assumptions on $f$ and $E$ that the base space $B$ is  nilpotent, so $B^\wedge_p$ is also $B\Z/p$-null by \Cref{thm:miller}. Furthermore, $p$-completion yields the indicated nilpotent fiber sequence. 
\end{proof}

\subsection{Stratification along fiber sequences}

Let $F \to E \to B$ be a fiber sequence where $E$ and $B$ are connected. The universal property of the pushout in commutative ring spectra provides a canonical map
\[
\xymatrix{\iota\colon C^*(E)\otimes_{C^*(B)} \F_p \ar[r] & C^*(F).}
\]
As explained in more detail in \cite[Proof of Lemma 5.6]{shamir_emss}, we can refine this map as follows: Without loss of generality, we can assume that $\Omega B$ is a topological group and that $F$ is a $\Omega B$-space (by using Kan's loop group functor \cite{kan_combinatorial_1958,goerss_simplicial_2010}, for example).  Consider the ring spectrum $R=C_*(\Omega B)$ and let $R \to \F_p$ be the canonical map, which allows us to view $\F_p$ as an $R$-module. Note that the Rothenberg--Steenrod construction shows that $\End_R(\F_p)\simeq C^*(B)$, see \cite[Section 4.22]{dgi_duality}. Note that as in the proof of \cite[Lemma 5.5]{shamir_emss} $C^*(F)$ is an $R$-module.  By \cite[Lemma 2.10]{dwyerwilkerson_transfer} there is an equivalence of $R$-modules $C^*(E)\simeq \Hom_R(\F_p,C^*(F))$ which fits into a map of $R$-modules 
\begin{equation}\label{eq:iota}
\xymatrix{C^*(E)\otimes_{C^*(B)} \F_p\simeq \Hom_{R}(\F_p,C^*(F))\otimes_{\End_R(\F_p)} \F_p \ar[r]^-{\iota'} & C^*(F),}
\end{equation}
where $\iota'$ is adjoint to the identity map on $\Hom_{R}(\F_p,C^*(F))$ via
\[
\Hom_{R}(\Hom_{R}(\F_p,C^*(F))\otimes_{\End_R(\F_p)} \F_p, C^*(F)) \simeq \Hom_{\End_R(\F_p)}(\Hom_{R}(\F_p,C^*(F)),\Hom_{R}(\F_p,C^*(F))).
\]
The composite constructed in \eqref{eq:iota} then coincides with $\iota$ as defined above.

\begin{defn}
We say that  the fiber sequence $F \to E \to B$  with $E$ and $B$ connected is of Eilenberg--Moore type (at the prime $p$) if the map $\iota$ is an equivalence.
\end{defn}

\begin{ex}
\begin{enumerate}
	\item Every trivial fibration $F \to F\times B \to B$ is of Eilenberg--Moore type in light of the equivalence $C^*(F\times B)\simeq C^*(F)\otimes_{\F_p} C^*(B)$.
	\item Recall that a fibration $F \to E \to B$ is simple if the action of $\pi_1(B)$ on the homotopy groups of $F$ is trivial and $\pi_1(B)$ is abelian. If $F \to E \to B$ satisfies an Eilenberg--Moore convergence theorem, such as when $\pi_1(B)$ acts nilpotently on $H_*(F)$ \cite{Dwyer1974Strong}, then it is of Eilenberg--Moore type. In particular, this is satisfied if the fibration is simple, e.g., a fibration of $H$-spaces and $H$-maps. 
	\item In \cite[Lemma 5.6]{shamir_emss}, the author shows that the morphism of $R$-modules $\iota$ is a $\F_p$-cellular approximation of $C^*(F)$ if $R$ is proxy-small, see also \cite[Proposition 4.10]{dgi_duality}. Moreover, in the proof of \cite[Proposition 5.8]{shamir_emss} it is shown that $C^*(F)$ is $\F_p$-cellular if $\pi_1(B)$ acts nilpotently on $H^n(F)$ for every $n$. Both of these conditions are satisfied if $\pi_1(B)$ is a finite $p$-group and $F$ of finite type, so $\iota$ is an equivalence in this case.
\end{enumerate}
\end{ex}

\begin{lem}\label{lem:chouinardfiber}
Suppose $F \to E \to B$ is a fibration of Eilenberg--Moore type with $H^*(B)$  finite. If $g\colon C^*(E) \to C^*(F)$ is the induced map on mod $p$ cochains, then both $g^*$ and $g^!$ are conservative. 
\end{lem}
\begin{proof}
Since $H^*(B)$ is finite, the induced map $f\colon C^*(B) \to \F_p$ is cosmall by \cite[Proposition 3.16]{dgi_duality}, i.e., $C^*(B) \in \Thick_{C^*(B)}(k)$. In particular, $f^*$ and $f^!$ are conservative by \Cref{lem:chouinardloc}. Moreover, because the fibration is of Eilenberg--More type, there is a pushout square of commutative ring spectra
\[
\xymatrix{C^*(B) \ar[r] \ar[d]_f & C^*(E) \ar[d]^g \\
k \ar[r] & C^*(F).}
\]
It follows that we are in the situation of \Cref{lem:chouinardbasechange}, so $g^*$ and $g^!$ are conservative.
\end{proof}

Recall that we write $P_{B\Z/p}$ for the $B\Z/p$-nullification functor; see \Cref{ssec:nullification} for the basic properties of $B\Z/p$-nullification.

\begin{lem}\label{lem:fisomorphism}
Let $F \xrightarrow{i} E \to B$ be a nilpotent fiber sequence of connected spaces  with mod $p$ cohomology of finite type and assume that $E$ is connected nilpotent and $B$ is connected and $B\mathbb Z/p$-null. Then the induced map $H^*(i)\colon H^*(E) \to H^*(F)$ is an $F$-monomorphism. If the fibration is simple, then $H^*(i)$ is an $F$-isomorphism. 
\end{lem}
\begin{proof}
The conditions of \Cref{lem:fiberseq} are satisfied, so the $p$-completion of $F \to E \to B$ results in a nilpotent fiber sequence
\begin{equation}\label{eq:pcompletedfiberseq}
\xymatrix{F_p^{\wedge} \ar[r]  & E_p^{\wedge} \ar[r] & B_p^{\wedge}}
\end{equation}
with $B\Z/p$-null base. The evaluation map then gives an equivalence $\map(B\Z/p,B_p^{\wedge}) \simeq B_p^{\wedge}$ since $B$ is connected. Now applying $\pi_*\map(BV,-)$ to the fiber sequence of \eqref{eq:pcompletedfiberseq} gives an exact sequence
\[
\xymatrix{\pi_1(B_p^{\wedge}) \ar[r] & [BV,F_p^{\wedge}] \ar[r]^-{\phi_V} & [BV,E_p^{\wedge}] \ar[r] & \pi_0(B_p^{\wedge})=*}.
\]
We see that $\phi_V$ is surjective and it identifies $[BV,E_p^{\wedge}]$ with the orbit space by the action of $\pi_1(B^\wedge_p)$ on $[BV,F_p^{\wedge}]$. If the fibration is simple, then $\phi_V$ is a bijection. Then the following natural composite is a epimorphism
\[
\xymatrix{\Hom_{\cK}(H^*(F),H^*(BV)) \simeq [BV,F_p^{\wedge}] \ar[r] & [BV,E_p^{\wedge}] \simeq  \Hom_{\cK}(H^*(E),H^*(BV)),}
\]
and even 	an isomorphism when the fibration is simple, as a consequence of Lannes' Theorem, see \Cref{thm:lannespi0comp}. Since this holds for all elementary abelian $p$-groups $V$, we deduce from \Cref{thm:fepimono} the conclusion for $H^*(E) \to H^*(F)$.
\end{proof}

\begin{prop}\label{prop:descentfiber}
Let $F \to E\to B$ be a simple nilpotent fiber sequence of Eilenberg--Moore type with $E$ nilpotent. Assume $H^*(F)$ and $H^*(E)$ are Noetherian and $H^*(B)$ finite, then the induced map $i\colon C^*(E)\to C^*(F)$ descends stratification, that is, if $\Mod_{C^*(F)}$ is  stratified, then so is $\Mod_{C^*(E)}$.
\end{prop}

\begin{proof}
The goal is to apply \Cref{thm:stratification} to the induced morphism of ring spectra
\[
\xymatrix{i \colon C^*(E) \ar[r] & C^*(F),}
\] 
which allows us to descend the stratification from $\Mod_{C^*(F)}$ to $\Mod_{C^*(E)}$. For this we need to check that $i$  satisfies Quillen lifting and that induction and coinduction along $i$ are conservative. 

Since $B$ is a connected nilpotent space with finite mod $p$ cohomology, it is $B\Z/p$-null by \Cref{thm:millersullivan}. Therefore, \Cref{lem:fisomorphism} implies that $i$ induces an $F$-isomorphism $H^*(E) \to H^*(F)$, so $i$ satisfies simple Quillen lifting by \Cref{lem:nilpotentCoker}, hence Quillen lifting by \Cref{lem:simplequillenlifting}. Moreover, \Cref{lem:chouinardfiber} applies to show that induction and coinduction along $i$ are conservative.
\end{proof}

\begin{rem}
Consider a trivial fibration $F \to F\times B \to B$ with $H^*(F)$ Noetherian, $H^*(B)$ finite, and suppose that $\Mod_{C^*(F)}$ stratified. Note that $i\colon C^*(F\times B)  \to C^*(F)$ is surjective on $\pi_*$, and in particular, an $F$-epimorphism. Using this instead of \Cref{lem:fisomorphism} in the proof of \cref{prop:descentfiber} yields the stratification of $\Mod_{C^*(F\times B)}$. 
\end{rem}

\subsection{Noetherian Hopf spaces}
Since $H$-spaces are simple, any fibration of $H$-spaces and $H$-maps is simple, and therefore of Eilenberg--Moore type. Moreover, simple spaces are in particular nilpotent, and hence $p$-good \cite[Proposition V.3.4]{bousfield_kan}. It follows that $C^*(X)\simeq C^*(X^\wedge_p)$, see \Cref{ssec:cochains}.

For the following, we recall that a discrete $p$-toral group $P$ is a group with a normal subgroup $P_0 \cong(\Z/p^{\infty})^r$ such that $\pi=P/P_0$ is a finite $p$-group. If the discrete $p$-toral group is abelian, then $P$ splits and it is isomorphic to a product $P \cong (\Z/p^{\infty})^r \times \pi$. In particular, the classifying space of any abelian $p$-discrete toral group is $\F_p$-equivalent to a compact abelian Lie group since $BP^\wedge_p \cong ((BS^1)^r \times B\pi)^\wedge_p$.

\begin{thm}\label{thm:hnoetherian}
If $X$ is a connected $H$-space with Noetherian mod $p$ cohomology, then $\Mod_{C^*(X)}$ is stratified.
\end{thm}
\begin{proof}
By \cite[Theorems 1.2,1.3]{bcs_deconstructing} (see also \cite{ccs_deconstructing}), the homotopy fiber of the $H$-map $X^\wedge_p\to (P_{B\Z/p}X)_p^{\wedge}$ is equivalent to $K(P,1)^\wedge_p$, where $P$ is an abelian discrete $p$-toral group. Moreover, the base space $(P_{B\Z/p}X)_p^{\wedge}$ is $B\mathbb Z/p$-null by \Cref{rem:h_space_null} and has finite mod $p$ cohomology. Note that since $X$ is an $H$-space, then $(P_{B\Z/p}X)_p^{\wedge}$ inherits the structure of an $H$-space and the nullification map is an $H$-map, since both nullification and completion are homotopy functors with commute with finite products. The corresponding fiber sequence is simple (as a fiber sequence of $H$-spaces) and of Eilenberg--Moore type, so the map 
\[
\xymatrix{\iota\colon C^*(X)\otimes_{C^*(P_{B\Z/p}X)}\F_p \ar[r] & C^*(K(P,1))}
\] 
is an equivalence. Since $C^*(K(P,1)) \simeq C^*(BA)$, where $A$ is a compact abelian Lie group, $\Mod_{C^*(K(P,1))}$ is stratified by \cite{bg_stratifyingcompactlie} or \cite{bchv1}, so we can apply \Cref{prop:descentfiber} to deduce the claimed result.
\end{proof}

As a consequence of \Cref{thm:tel_conj}, the the telescope conjecture holds for $\Mod_{C^*(X)}$ for any connected $H$-space $X$ with Noetherian mod $p$ cohomology.

\subsection{An explicit example}\label{sec:s3cover}

We conclude this section with an explicit example given by $X = S^3 \langle 3 \rangle$, the $3$-connected cover of $S^3$. This space fits in a principal fibration $BS^1\xrightarrow{i} S^3 \langle 3 \rangle\xrightarrow{j}  S^3$ induced by the fundamental class $S^3 \to K(\Z,3)$. The mod $p$ cohomology of $S^3 \langle 3 \rangle$ is 
\[
H^*(S^3 \langle 3 \rangle)\cong \F_p[x_{2p}]\otimes \Lambda (y_{2p+1}).
\]
In cohomology,  $H^*(Bi)\colon H^*(S^3\langle 3 \rangle)\to H^*(BS^1)$ sends $x_{2p}$ to $x^p$, so this is an $F$-isomorphism. Then $C^*(Bi)\colon C^*(S^3\langle 3 \rangle)\to C^*(BS^1)$ satisfies simple Quillen lifting by \Cref{lem:nilpotentCoker}. Since $S^3$ is simply connected and has finite cohomology, \Cref{lem:chouinardfiber} implies that $C^*(i)\colon C^*(S^3 \langle 3 \rangle) \to C^*(BS^1)$ is biconservative. As a consequence of \Cref{thm:stratification}, $\Mod_{C^*(S^3\langle 3 \rangle)}$ is stratified.

Note that the morphism $C^*(i)\colon C^*(S^3\langle 3\rangle) \to C^*(BS^1)$ is neither finite nor split. Indeed, because all spaces involved are simply connected, the fiber sequence $\Omega S^3 \to BS^1 \to S^3\langle 3\rangle$ is of Eilenberg--Moore type, so we have a pushout square of commutative ring spectra 
\[
\xymatrix{C^*(S^3\langle 3\rangle) \ar[r]^-{C^*(i)} \ar[d] & C^*(BS^1) \ar[d] \\
k \ar[r]_-{C^*(j)} & C^*(\Omega S^3).} 
\]
Since $H^*(\Omega S^3)$ is infinite-dimensional, the morphism $C^*(j)$ is not finite, and thus neither is $C^*(i)$. Further, the ring $H^*(BS^1;\F_p)$ does not contain nilpotent elements, so the previous cohomology calculation implies that $C^*(i)$ cannot admit a retraction even on homotopy groups. 

This means that none of the known techniques for establishing costratification (for example, \cite[Proposition~3.20 or Theorem~3.27]{bchv1}) apply to the morphism $C^*(i)\colon C^*(S^3\langle 3\rangle) \to C^*(BS^1)$, so we currently do not know whether $\Mod_{C^*(S^3\langle 3\rangle)}$ is costratified.

\section{Chouinard's condition and stratification}\label{sec:stratspaces}

In order to use the descent results of \Cref{sec:quillen} to study stratifications for $\Mod_{C^*(X)}$ for general topological spaces $X$ with Noetherian mod $p$ cohomology, we need to first find a candidate morphism $\rho_X\colon C^*(X) \to S$ with $\Mod_S$ stratified. The goal of this section is to construct such a morphism $\rho_X$ using Lannes' theory and to then deduce from Rector's generalization of Quillen's stratification theorem that $\rho_X$ satisfies Quillen lifting. 

\subsection{Rector's category}

Given the distinguished role played by the cohomology of elementary abelian $p$-groups $V$ for detecting when a morphism in $\cK$ is an $F$-isomorphism, it is convenient to organize $\mathcal A_p$-algebra morphisms from a fixed $K \in \cal{K}$ to $H^*(BV)$ in a category $\cR(K)$. This category was introduced by Rector~\cite{rector_quillenstrat}, generalizing previous work of Quillen~\cite{quillen_stratification} on the cohomology of compact Lie groups. 

\begin{defn}
Let $K$ be a Noetherian unstable $\cA_p$-algebra. The Rector category $\cR(K)$ of $K$ has as objects pairs $(V,\varphi)$, where $V$ is an elementary abelian $p$-group and $\varphi\colon K\rightarrow H^*(BV)$ is a morphism of unstable algebras which is finite, i.e., $H^*(BV)$ is a finitely generated $K$-module via $\varphi$. A morphism $f\colon (V,\varphi) \to (V',\varphi')$ between objects in $\cR(K)$ is given by a monomorphism $f\colon V \to V'$ such that $H^*(Bf)\circ \varphi' = \varphi$.
\end{defn}

\begin{rem}
This category is in fact the opposite of the so-called fundamental category of $K$ as defined by Rector, cf.~Definition 1.2 in \cite{rector_quillenstrat}; by now, it is more conventional to work with $\cR(K)$ instead of Rector's original category $\cR(K)^{\op}$. 
\end{rem}
\begin{ex}
	Let $G$ be a compact Lie group. Then Lannes' theory shows that  $\cR(H^*(BG))$ is equivalent to the Quillen category $\cal{A}_p(G)$ of $G$ introduced in \Cref{sec:kmkm}, see the discussion in \cite[Section 5.3]{hls_localization}. 
\end{ex}

If $K \in \cK$ is Noetherian, Rector showed that $\mathcal R(K)$ is equivalent to a finite category, since isomorphism classes of objects are determined by $\cal A_p$-invariant prime ideals and $K$ has finitely many invariant prime ideals. Moreover, there is only a finite set of morphisms since the set $\Hom(V,V')$ of abelian group homomorphisms is finite for two elementary abelian $p$-groups $V,V'$. The functor from $\cR(K)$ to $\cK$ which assigns $H^*(BV;\F_p)$ to $(V,\varphi)$ induces a natural homomorphism
\[
\xymatrix{r_K\colon K \ar[r] & \varprojlim_{\mathcal R(K)} H^*(BV).}
\]
Rector~\cite{rector_quillenstrat} and Broto--Zarati~\cite{brotozarati_steenrod} prove $r_K$ provides a close approximation to $K$:

\begin{thm}[Rector, Broto--Zarati]\label{thm:rector_bz}
Let $K$ be a Noetherian unstable $\cA_p$-algebra, then $r_K$ is an $F$-isomorphism. 
\end{thm}
In his work on equivariant cohomology Quillen \cite[Section 10 and 11]{quillen_stratification} shows how this gives a decomposition of the prime ideal spectrum of $H^*(BG)$. Rector proves a similar result \cite[Theorem 1.7]{rector_quillenstrat}. Note that Rector states the result in terms of the maximal ideal spectrum instead of the prime ideal spectrum - as explained after Proposition 11.2 of \cite{quillen_stratification}, the results also hold for the prime ideal spectrum, with the exception that the `Weyl group' (denoted by $W_{\phi}$ below) no longer acts freely, but rather only transitively. In order to state these results, we need to introduce some notation, similar to that used in \Cref{ssec:hierarchy}. 

For an elementary abelian group $E$, we recall that $\cal{V}_E = \Spec^h(H^*(BE))$ and that  $\cV_{E}^+ = \cV_{E} \setminus \bigcup_{E' < E} \res_{E}^{E'}\cV_{E'}$, where $\res_{E}^{E'}\colon \cV_{E'} \to \cV_{E}$ is the restriction map. For $K \in \cal{K}$, we will write $\cal{V}_{K}=\Spec^h(K)$. If $\phi \colon K \to H^*(BE)$ is a morphism in $\cal{R}(K)$, then $\res_{\phi}$ denotes the corresponding restriction map on varieties. We let 
\[
\cal{V}_{\phi,E} = \res_{\phi}(\cal{V}_E) \quad \text{ and } \quad \cal{V}^+_{\phi,E} = \res_{\phi}(\cal{V}_E^+). 
\]
Finally, we let $W_{\phi}$ denote the group of automorphisms of $(E,\phi)$ in $\cal{R}(K)$, and let $\cal{E}(K)$ denote a set of representatives for the isomorphism classes of objects in $\cal{R}(K)$. 
\begin{thm}[Rector]\label{thm:rectorstrong}
  Let $K$ be a Noetherian unstable $\cA_p$-algebra, then: 
  \begin{enumerate}
    \item There is a decomposition of varieties
    \[
\cal{V}_{K} = \coprod_{(E,\phi) \in \cal{E}(K)} \cV_{\phi,E}^+.
    \]
    \item $W_{\phi}$ acts transitively on $\cal{V}_{E}^+$, and $\cal{V}_{\phi,E}^+ \cong \cal{V}_E^+/W_{\phi}$. 
    \item $\cal{V}_{\phi,E} \subseteq \cal{V}_{\phi',E}$ if and only if there is a morphism $(E',\phi') \to (E,\phi)$ in $\cal{R}(K)$. 
  \end{enumerate}
\end{thm}
We would like to discuss the situation when $K$ is the mod $p$ cohomology of a topological space $X$. For suitable $X$, the objects and morphisms in $\cR(K) = \cR(H^*(X))$ can be realized by maps between spaces. Indeed, let $X$ be a $p$-good connected topological space such that $H^*(X)$ is Noetherian. Consider Rector's category $\cR(H^*(X))$ associated to the Noetherian unstable $\cA_p$-algebra $H^*(X)$; for simplicity, we write this simply as $\cR(X)$. By \Cref{thm:lannespi0comp} any map of unstable $\cA_p$-algebras $\phi \colon H^*(X) \to H^*(BV)$ with $V$ an elementary abelian $p$-group is realized by a unique map up to homotopy $\widetilde \phi \colon BV\rightarrow X^\wedge_p$, i.e., $H^*(\widetilde \phi )=\phi$. Since $X$ is $p$-good, this provides a map 
\[
\xymatrix{\rho_{X}^{\phi}\colon C^*(X) \simeq C^*(X^\wedge_p) \ar[r]^-{C^*(\widetilde \phi)} & C^*(BV)}
\]
with $\pi_*\rho_{X}^{\phi} = \phi$. We record this result here. 
\begin{lem}\label{prop:lift}
  Let $X$ be a $p$-good connected topological space $X$, then for any pair $(V,\phi) \in \cal{R}(X)$ there exists $\rho_X^{\phi} \colon C^*(X) \to C^*(BV)$ as above with $\pi_*\rho_{X}^{\phi} = \phi$. Moreover, for any morphism $f \colon (V,\phi) \to (V',\phi')$ in $\cal{R}(X)$, we obtain a commutative diagram as follows
  \[
\begin{tikzcd}
	C^*(X) \arrow{d}[swap]{\rho_X^{\phi'}} \arrow{r}{\rho_X^{\phi}} & C^*(BV) \\
	C^*(BV') \arrow{ur}[swap]{C^*(Bf)},
\end{tikzcd}
\]
which realizes the commutative diagram
\[
\begin{tikzcd}
	H^*(X) \arrow{d}[swap]{{\phi'}} \arrow{r}{{\phi}} & H^*(V) \\
	H^*(V') \arrow{ur}[swap]{H^*(Bf)}
\end{tikzcd}
\]
on homotopy. 
\end{lem}
\begin{proof}
Everything but the claim that we can realize the commutative diagram on homotopy by maps of cochains has been shown. This follows because the maps produced by applying \Cref{thm:lannespi0comp} are unique up to homotopy; since $\phi \cong H^*(Bf)\circ \phi'$, there is a unique map $BV \to X^{\wedge}_p$ realizing $\phi \cong H^*(Bf)\circ \phi'$ on cohomology which factors through $Bf \colon BV \to BV'$. Taking cochains, we deduce that $C^*(Bf) \circ \rho_X^{\phi'} \simeq \rho_X^{\phi}$. 
\end{proof}
 Let $\cE(X)$ be a set of representatives of isomorphism classes of objects in $\cR(X)$, i.e., $\cE(X) = \cE(H^*(X))$. We thus obtain a map of commutative ring spectra
\begin{equation}\label{eq:chouinardmap}
\xymatrix{\rho_X := \prod_{(E,\phi)\in\cE(X)}\rho_{X}^{\phi} \colon C^*(X)\ar[r] & \prod_{(E,\phi)\in \cE(X)} C^*(BE)}
\end{equation}
for any $p$-good connected topological space $X$. This will be our candidate for descending stratification.

\subsection{Quillen lifting for cochains of spaces}
We now show that the morphism $\rho_X$ constructed in the previous section always satisfies Quillen lifting. We continue to assume that $X$ is a $p$-good connected topological space with Noetherian mod $p$ cohomology. 

\begin{prop}\label{prop:autoql}
The morphism $\rho_X\colon C^*(X) \to \prod_{(E,\phi) \in \cE(X)} C^*(BE)$ defined in \eqref{eq:chouinardmap} satisfies simple Quillen lifting. 
\end{prop}
\begin{proof}
In order to prepare for the proof, we start by introducing some notation. Let $\cV_X = \Spec^h(H^*(X))$ and $\cV_{\cE(X)} = \Spec^h(\prod_{E \in \cE(X)}H^*(BE))$, and write $\rho = \rho_X$. By \Cref{thm:rectorstrong} there is a commutative diagram
\[
\xymatrix{\cV_{\cE(X)} \ar[r]^-{\res_{\rho}} & \cV_{X} \\
\coprod_{(E,\phi) \in \cE(X)}\cV_{E} \ar[r] \ar[u]^{\cong} & \coprod_{(E,\phi) \in \cE(X)}\cV_{\phi,E}^+. \ar[u]^{\fS}_{=} }
\]
Let $\fp \in \cV_X$ be a prime ideal. Since $\fS$ is a bijection, we may choose a pair $(E,\phi) \in \cE(X)$ in which $\fp$ originates: this means that there exists a (not necessarily unique) prime ideal $\cO(\fp) \in \cV_{E}^+ \subseteq \cV_{E}$ with $\fS([\cO(\fp)]) = \fp$, where $[\cO(\fp)]$ denotes the image of $\cO(\fp)$ under the quotient map $\cV_E^+ \to \cV_{\phi,E}^+$ of \Cref{thm:rectorstrong}(2). Since $\Mod_{C^*(BE)}$ is stratified for all elementary abelian groups $E$ by \cref{thm:stratcompactLie}, it remains to check condition ($\clubsuit$) of \cref{ssec:criterionsql} for this choice of $\cO$.

To this end, let $(E',\phi') \in \cE(X)$ and $\fq \in \cV_{E'}$ be a prime ideal such that $\res_{\rho_X^{\phi}}(\fq) = \fp$. By \cref{thm:rectorstrong}, there exists a monomorphism $\zeta\colon (E,\phi) \to (E',\phi')$ in $\cR(X)$ such that $\fq$ is the image of $\cO(\fp)$ under the associated restriction map $\res_{\zeta}\colon \cal{V}_{E} \to \cal{V}_{E'}$. The map $\zeta$ fits into a commutative diagram of unstable algebras over the Steenrod algebra, as displayed on the left, 
\[
\xymatrix{H^*(X) \ar[r]^-{H^*(\phi)} \ar[rd]_-{H^*(\phi')} & H^*(BE) \ar@{<-}[d]^{H^*(\zeta)} & C^*(X) \ar[r]^-{C^*(\phi)} \ar[rd]_-{C^*(\phi')} & C^*(BE) \ar@{<-}[d]^{C^*(\zeta)} \\ 
& H^*(BE')  & & C^*(BE'),}
\]
which by virtue of \Cref{prop:lift} lifts to the commutative diagram on the right. This verifies that $\cO$ is a suitable section to $\res_{\rho}$ satisfying condition ($\clubsuit$), as desired.
\end{proof}

\begin{rem}
One might wonder if \cref{cor:simplequillenliftingk1x} is a special case of this proposition. This is not readily the case, because we do not know whether $X = BG$ for $G \in \cal{K}_1\cal{X}$ is $p$-good. However, this proposition does apply to $X = BG$ for $G$ a compact Lie group to show that $C^*(BG) \to \prod_{E\in\cE(G)}C^*(BE)$ satisfies simple Quillen lifting. 
\end{rem}

\subsection{Chouinard's condition for cochains on spaces}\label{ssec:chouinardspaces}

\begin{defn}
A $p$-good connected topological space $X$ with Noetherian mod $p$ cohomology is said to satisfy Chouinard's condition if the map $\rho_X$ of \eqref{eq:chouinardmap} is biconservative, i.e., if induction and coinduction along $\rho_X$ are conservative.
\end{defn}
\begin{rem}
\Cref{prop:reschouinard} exhibits a necessary condition for $\rho_X$ to be conservative: $\res_{\rho_X}$ needs to be surjective. This is indeed the case and follows from the decomposition of the variety $\cal{V}_{H^*(X)}$ in \Cref{thm:rectorstrong}(1) and (2). 
\end{rem}

\begin{thm}\label{thm:rectorstrat}
Let $X$ be a $p$-good connected space with Noetherian mod $p$ cohomology, then $\Mod_{C^*(X)}$ is stratified if and only if $X$ satisfies Chouinard's condition.
\end{thm}
\begin{proof}
Suppose that $\rho_X$ is biconservative. \Cref{rem:products,thm:stratcompactLie} show that the category $\Mod_{\prod_{(E,\phi) \in \cR(X)} C^*(BE)} \simeq \prod_{(E,\phi) \in \cR(X)} \Mod_{C^*(BE)}$ is stratified. Since Quillen lifting for $\rho_X$ has been proven in \Cref{prop:autoql}, we can descend stratification along $\rho_X$ by \Cref{thm:stratification}, so $\Mod_{C^*(X)}$ is stratified.

The \emph{only if} direction follows from \Cref{prop:reschouinard} and Rector's strong form of Quillen stratification given in \Cref{thm:rectorstrong}. 
\end{proof}

\begin{ex}\label{ex:Chouinardfinitecomplex}
Let $X$ be a connected finite complex. Then Rector's category consists only of the trivial subgroup: indeed, for any elementary abelian $p$-group $V$ and any map $f\colon BV\to X$, Lannes' theory says that $H^*(f) \colon H^*(X) \to H^*(BV)$ is trivial. This also follows from Miller's solution of the Sullivan conjecture \cite{miller_sullivan}, since any map $BV \to X$ is nullhomotopic. The map $H^*(f)$ is thus finite if and only if $V=\{e\}$, the trivial group. Moreover, the canonical map $C^*(X) \to \F_p$ is biconservative. This follows from the same argument in \Cref{lem:chouinardfiber}: because $H^*(X)$ is finite, we have $C^*(X) \in \Thick_{C^*(X)}(\F_p)$, and so \Cref{lem:chouinardloc} applies. We deduce that $C^*(X)$ is stratified.
\end{ex}

\begin{rem}
	More generally, suppose $X$ is a connected finite CW-complex, and $R$ a Noetherian commutative ring spectrum so that $\pi_*C^*(X;R)$ is Noetherian. There is a splitting
\[
\xymatrix{R \ar[r] &  C^*(X;R) \ar[r]^-s & R}
\]
of $R$-modules, where the morphism $s$ is given by evaluating at a basepoint $\ast \in X$. We have that $C^*(X;R) \in \Thick_{C^*(X;R)}(R)$ by induction on the number of cells of $X$. It follows that $s$ is biconservative. Moreover, the splitting implies that $\pi_*s\colon \pi_*C^*(X;R) \to \pi_*R$ is surjective, and hence simple Quillen lifting is satisfied by \Cref{lem:simplequillenlifting}. It follows from \Cref{thm:stratification} that $\Mod_{C^*(X;R)}$ is stratified. 
\end{rem}

\subsection{Chouinard's condition for $H$-spaces with Noetherian mod $p$ cohomology}

Finally, we verify that $H$-spaces with Noetherian mod $p$ cohomology satisfy Chouinard's condition. In particular, this demonstrates that the stratification of $H$-spaces with Noetherian mod $p$ cohomology (\Cref{thm:hnoetherian}) fits into the more general framework of this section. 

\begin{thm}\label{thm:nhchouinard}
Let $X$ be an $H$-space with Noetherian mod $p$ cohomology, then $X$ satisfies Chouinard's condition, i.e., the map $\rho_X$ constructed in \eqref{eq:chouinardmap} is biconservative. 
\end{thm}
\begin{proof}
As in the proof of \Cref{thm:hnoetherian}, \cite{bcs_deconstructing} provides a nilpotent fiber sequence of Eilenberg--Moore type
\[
\xymatrix{K(P,1)^\wedge_p \ar[r]^-i & X^\wedge_p \ar[r]^-j & (P_{B\Z/p}X)_p^{\wedge},}
\]
in which $P$ is an abelian discrete $p$-toral group and $(P_{B\Z/p}X)_p^{\wedge}$ is a $B\mathbb Z/p$-null nilpotent space with finite mod $p$ cohomology. We recall that $C^*(K(P,1)^{\wedge}_p) \simeq C^*(BA)$ for $A$ an abelian compact Lie group. We may thus replace $K(P,1)$ by $BA$. The proof of \Cref{lem:fisomorphism} shows that the induced map $H^*(i)\colon H^*(X) \to H^*(BA)$ is an $F$-isomorphism. Because $H^*(X)$ and $H^*(BA)$ are both finitely generated $\F_p$-algebras, $H^*(i)$ is even a uniform $F$-isomorphism, see the remark after A.2 of \cite{Henn1998Unstable}, and hence $H^*(i)$ induces an isomorphism of Rector categories $\cR(BA) \cong \cR(X)$, by \cite[Proposition 1.5]{rector_quillenstrat}.  In particular, passage to isomorphism classes gives a bijection $\cE(BA) \cong \cE(X)$ induced by precomposing with $H^*(i)$.

Given $(E,\phi)\in \cE(BA)$, by \Cref{prop:lift}, there is a map $\rho^\phi_{BA}\colon C^*(BA)\to C^*(BE)$ such that $\pi_*(\rho^\phi_{BA})=\phi$. Then, for the corresponding object $(E,\phi\circ H^*(i))\in \cE(X)$, the map $\rho^\phi_{BA} \circ C^*(i)$  is a model for $\rho^{\phi\circ H^*(i)}_{BA}$, since $\pi_*(\rho^\phi_{BA} \circ C^*(i))=\phi \circ H^*(i)$. Therefore, the map $\rho_X$ defined in \Cref{eq:chouinardmap} can be taken to be $\rho_{BA}\circ C^*(i)$, i.e., the following diagram commutes:
\[
\xymatrix{C^*(X) \ar[r]^-{\rho_X} \ar[d]_-{C^*(i)} & \prod_{(E,\phi\circ H^*(i)) \in \cE(X)}C^*(BE) \ar@{=}[d]^-{ } \\
C^*(BA) \ar[r]_-{\rho_{BA}} & \prod_{(E,\phi) \in \cE(BA)}C^*(BE).}
\]
\Cref{lem:chouinardfiber} applies to show that $C^*(i)$ is biconservative, while $\rho_{BA}$ is biconservative by \cite[Theorem 3.1(i)]{bg_stratifyingcompactlie}, as observed in \Cref{ex:ChouinardcompactLie}. It follows from \Cref{lemma:chouinard-transitivity} that $\rho_X$ is biconservative as well. 
\end{proof}

\begin{proof}[Alternative proof of \Cref{thm:hnoetherian}]
Any $H$-space is simple and hence $p$-good, so by \Cref{thm:rectorstrat} it suffices to show that $X$ satisfies Chouinard's condition. This is the content of \Cref{thm:nhchouinard}; the claim follows.
\end{proof}

\begin{ex}
We return to the example of $X = S^3\langle 3 \rangle$ considered in \Cref{sec:s3cover}. We recall that $S^3\langle 3 \rangle$ fits into a principal fibration $BS^1 \xr{i} S^3 \langle 3 \rangle \xr{j} S^3$. Rector's category $\cal{R}(S^3\langle 3 \rangle)$ consists of a single non-trivial object, given by the inclusion $\iota\colon \Z/p \to  S^1$. Let $f \colon C^*(S^3 \langle 3) \rangle \to C^*(B\Z/p)$ denote the composite $C^*(B\iota) \circ C^*(i)$. This map is an $F$-isomorphism, and so $f$ satisfies simple Quillen lifting by \Cref{lem:simplequillenlifting}. As we have already seen, $C^*(i)$ is biconservative, and so is $C^*(B\iota)$ by \Cref{ex:ChouinardcompactLie}. It follows that $f$ and hence $\rho_{S^3\langle 3 \rangle}$ are biconservative, and we again deduce stratification of $C^*(S^3\langle 3 \rangle)$.  
\end{ex}

\biblio
\bibliography{duality}\bibliographystyle{alpha}
\end{document}